\documentclass[12pt]{article}% Math and Physical Sciences Reference Style
\usepackage{color}
\usepackage{amsmath}
\usepackage{extarrows}
\usepackage{geometry}
\usepackage{authblk} % \affil
\usepackage{hyperref} % ref in PDF
\usepackage{color}
\usepackage{ntheorem}
\usepackage{enumerate}
\usepackage[mathscr]{euscript}
\usepackage[utf8]{inputenc}
\def\q{\hfill\rule{1ex}{1ex}}
\def\0{\emptyset}

\def\q{\hfill\rule{1ex}{1ex}}

\newtheorem{theorem}{Theorem}[section]

\newtheorem{lemma}[theorem]{Lemma}
\newtheorem{claim}[theorem]{Claim}

\newtheorem{cor}[theorem]{Corollary}

\newenvironment{proof}{{\noindent\it Proof.}}{\hfill $\square$\par}

\usepackage{graphicx}

\usepackage{
	amsmath,			% Math Environments
	amssymb,			% Extended Symbols
	enumerate,		    % Enumerate Environments
	graphicx,			% Include Images
	lastpage,			% R\part{title}eference Lastpage
	multicol,			% Use Multi-columns
	multirow,			% Use Multi-rows
	pifont,			    % For Checkmarks
}

\usepackage[numbers]{natbib}

\newcounter{cases}
\newcounter{subcases}[cases]
\newenvironment{mycase}
{
    \setcounter{cases}{0}
    \setcounter{subcases}{0}
    \newcommand{\case}
    {
        \par\indent\stepcounter{cases}\textbf{Case \thecases.}
    }
    
}
{
    \par
}
\renewcommand*\thecases{\arabic{cases}}

\begin{document}
\allowdisplaybreaks
% --- PAPER INFO ---

\title{The maximum sum of sizes of non-empty cross $L$-intersecting families}
\author{
    {\small\bf Xiamiao Zhao}\thanks{email:  zxm23@mails.tsinghua.edu.cn}\quad
    {\small\bf Haixiang Zhang}\thanks{email:  zhang-hx22@mails.tsinghua.edu.cn}\quad% {\small\bf Mengyu Cao}\thanks{email:  myucao@ruc.edu.cn}\quad
    % {\small\bf Zequn Lv}\thanks{email:  lvzq19@mails.tsinghua.edu.cn}\quad
    {\small\bf Mei Lu}\thanks{email: lumei@tsinghua.edu.cn}\\
    {\small Department of Mathematical Sciences, Tsinghua University, Beijing 100084, China.}\\
    % {\small Institute for Mathematical Sciences, Renmin University of China, Beijing 100086, China.}\\
}

%\author{Yichen Wang}
\date{}

\maketitle\baselineskip 16.3pt

\begin{abstract}
Let $n$, $r$, and $k$ be positive integers such that $k, r \geq 2$,  $L$  a non-empty subset of $[k]$, and  $\mathcal{F}_i \subseteq \binom{[n]}{k}$ for $1 \leq i \leq r$. We say that non-empty families $\mathcal{F}_1, \mathcal{F}_2, \ldots, \mathcal{F}_r$ are $r$-cross $L$-intersecting if $\left| \bigcap_{i=1}^r F_i \right| \in L$ for every choice of $F_i \in \mathcal{F}_i$ with $1 \leq i \leq r$. They are called  pairwise cross $L$-intersecting if $|A \cap B| \in L$ for all $A \in \mathcal{F}_i$, $B \in \mathcal{F}_j$ with $i \neq j$.
If $r=2$, we  simply say  cross $L$-intersecting instead of $2$-cross $L$-intersecting or  pairwise cross $L$-intersecting. In this paper, we determine the maximum possible sum of sizes of non-empty cross $L$-intersecting families $\mathcal{F}_1$ and $\mathcal{F}_2$ for all admissible $n$, $k$, and $L$, and we characterize all the extremal structures.  We also establish the maximum value of the sum of sizes of families $\mathcal{F}_1, \dots, \mathcal{F}_r$ that are both pairwise cross $L$-intersecting and $r$-cross $L$-intersecting, provided $n$ is sufficiently large and $L$ satisfies certain conditions. Furthermore, we characterize all such families attaining the maximum total size.
\end{abstract}

%\textbf{AMS classification: }\textit{05C75, 05C65, 05C05}\vskip 0.3cm

\vskip.3cm

\section{Introduction}
Set $\binom{X}{k}$ to denote the family of all $k$-subsets of $X$. In particular, for a positive integer $n$, let $[n]$ denote the set $\{1,2,\dots,n\}$, and $[k,n] = \{k,\dots,n\}$ for $k \leq n$. For two positive integers $a,b$, we set $\binom{a}{b} = 0$ if $b > a$.

Let $n$, $t$, and $k$ be positive integers such that $k \geq 2$ and $1 \le t \le k-1$. A family $\mathcal{A} \subseteq \binom{[n]}{k}$ is called {\em $t$-intersecting} if $|A \cap B| \geq t$ for all $A,B \in \mathcal{A}$. When $t = 1$, we say $\mathcal{A}$ is intersecting rather than $t$-intersecting. The celebrated Erd\H{o}s--Ko--Rado theorem \cite{erdos1961intersection} states that for $n \geq 2k$, a full star is the unique intersecting $k$-uniform family of maximum size.
\begin{theorem}[\cite{erdos1961intersection}]
    Let $n \geq 2k$, and let $\mathcal{A} \subseteq \binom{[n]}{k}$ be an intersecting family. Then
    \[
    |\mathcal{A}| \leq \binom{n-1}{k-1}.
    \]
    Moreover, when $n > 2k$, equality holds if and only if $\mathcal{A} = \{A \in \binom{[n]}{k} : i \in A\}$ for some $i \in [n]$.
\end{theorem}

The Erd\H{o}s--Ko--Rado theorem is widely regarded as a cornerstone of extremal combinatorics and has led to numerous generalizations and applications. More generally, a family $\mathcal{A}$ is said to be {\em $t$-intersecting} if the size of the intersection of any two distinct sets lies in the interval $[t, k]$. Viewing $[t,k]$ as a subset of $[k]$, one can naturally extend the notion of $t$-intersection to that of $L$-intersection.

Specifically, for a non-empty set $L \subseteq [k]$, a family $\mathcal{A}$ is called {\em $L$-intersecting} if $|A \cap B| \in L$ for all distinct $A, B \in \mathcal{A}$. The problem of determining the maximum size of an $L$-intersecting family has been extensively studied in extremal set theory; see, for example, \cite{frankl2016invitation}. The celebrated Deza--Erd\H{o}s--Frankl theorem provides a general upper bound for such families and describes the rough structure of $L$-intersecting families that are nearly optimal in size.

\begin{theorem}[\cite{articlefrankl1978}]\label{thm: frankl intersect}
    Let $k \geq 3$, $n \geq 2^k k^3$, and $L \subseteq [0,k]$. If a family $\mathcal{F} \subseteq \binom{[n]}{k}$ is $L$-intersecting, then $|\mathcal{F}| \leq \prod_{\ell \in L} \frac{n - \ell}{k - \ell}$. Moreover, there exists a constant $C = C(k,L)$ such that every $L$-intersecting family $\mathcal{F}' \subseteq \binom{[n]}{k}$ with $|\mathcal{F}'| \geq C n^{|L|-1}$ satisfies $\left| \bigcap_{A \in \mathcal{F}'} A \right| \geq \min L$.
\end{theorem}

For non-empty families $\mathcal{A}_1, \mathcal{A}_2, \dots, \mathcal{A}_r \subseteq \binom{[n]}{k}$, they are said to be {\em pairwise cross $t$-intersecting} if $|A \cap B| \geq t$ for all $A \in \mathcal{A}_i$, $B \in \mathcal{A}_j$ with $i \neq j$. They are called {\em $r$-cross $t$-intersecting} if $\left| \bigcap_{i=1}^r A_i \right| \geq t$ for all $A_i \in \mathcal{A}_i$.

This paper generalizes the study of cross $t$-intersecting families to the cross $L$-intersecting families. For a non-empty set $L \subseteq [0,k]$, families $\mathcal{A}_1, \mathcal{A}_2, \dots, \mathcal{A}_r \subseteq \binom{[n]}{k}$ are called {\em pairwise cross $L$-intersecting} if $|A \cap B| \in L$ for all $A \in \mathcal{A}_i$, $B \in \mathcal{A}_j$ with $i \neq j$. They are called {\em $r$-cross $L$-intersecting} if $\left| \bigcap_{i=1}^r A_i \right| \in L$ for all $A_i \in \mathcal{A}_i$. Note that when $L = [t, k]$, pairwise cross $L$-intersecting (resp.\ $r$-cross $L$-intersecting) reduces to pairwise cross $t$-intersecting (resp.\ $r$-cross $t$-intersecting). When $r = 2$, we simply say {\em cross $L$-intersecting} or {\em cross $t$-intersecting} instead of $2$-cross $L$-intersecting or $2$-cross $t$-intersecting. In what follows, let $L \subseteq [0,k]$ be a non-empty set, and define $k-L = \{k - i : i \in L\}$. For $\mathcal{A} \subseteq \binom{[n]}{k}$, let $\overline{\mathcal{A}} = \binom{[n]}{k} \setminus \mathcal{A}$.

Wang and Zhang \cite{WANG2013129} determined the maximum total size of two cross $t$-intersecting families of finite sets.

\begin{theorem}[\cite{WANG2013129}]
    Let $n, a, b, t$ be positive integers such that $n \geq 4$, $a, b \geq 2$, $t < \min\{a,b\}$, $a + b < n + t$, $(n,t) \neq (a+b,1)$, and $\binom{n}{a} \leq \binom{n}{b}$. If non-empty families $\mathcal{A} \subseteq \binom{[n]}{a}$ and $\mathcal{B} \subseteq \binom{[n]}{b}$ are cross $t$-intersecting, then
    \[
    |\mathcal{A}| + |\mathcal{B}| \leq \binom{n}{b} - \sum_{i=0}^{t-1} \binom{a}{i} \binom{n-a}{b-i} + 1.
    \]
\end{theorem}

The main method in the proof of the above theorem involves analyzing a bipartite graph defined in \cite{WANG2013129}. Using a similar approach, we generalize their result in a specific setting. Let $\mathcal{A}, \mathcal{B} \subseteq \binom{[n]}{k}$. Note that when $k \leq n < 2k$ and $[2k-n, k] \cap L = \emptyset$, any two $k$-sets in $\binom{[n]}{k}$ intersect in at least $2k - n$ elements, which forces one of $\mathcal{A}$ or $\mathcal{B}$ to be empty.  Hence, we first make the following remark.

\vskip 0.2cm
\noindent \textbf{Remark.} If $n < 2k$ and $[2k - n, k] \cap L = \emptyset$, then there exist no non-empty cross $L$-intersecting families $\mathcal{A}, \mathcal{B} \subseteq \binom{[n]}{k}$.

We now present our first result.

\begin{theorem}\label{thm: 2 cross L intersecting}
    Let $n, k$ be positive integers with $n \geq k \geq 2$, and $\emptyset \neq L \subseteq [0,k]$. %Assume $[2k - n, k] \cap L \neq \emptyset$ if $n < 2k$.
    Let $\mathcal{A}, \mathcal{B} \subseteq \binom{[n]}{k}$ be non-empty cross $L$-intersecting families.

    \noindent \textbf{\rm(i)} If $n \geq 2k$ and $L = [0,k]$, or if $n < 2k$ and $[2k - n, k] \subseteq L$, then
    \[
    |\mathcal{A}| + |\mathcal{B}| \leq 2 \binom{n}{k},
    \]
    and equality holds if and only if $\mathcal{A} = \mathcal{B} = \binom{[n]}{k}$.

    \noindent \textbf{\rm(ii)} If $n > 2k$ and $L \neq [0,k]$, or $n = 2k$ and $L \notin \{[0,k], k-L\}$, or $n < 2k$ and $[2k - n, k] \cap L \notin \{\emptyset, [2k - n, k]\}$, then
    \[
    |\mathcal{A}| + |\mathcal{B}| \leq \sum_{i \in L} \binom{k}{i} \binom{n-k}{k-i} + 1,
    \]
    and equality holds if and only if (up to isomorphism) one of the following holds:
    \begin{enumerate}
        \item $\mathcal{A} = \{[k]\}$, $\mathcal{B} = \left\{ F \in \binom{[n]}{k} : |F \cap [k]| \in L \right\}$;
        \item $n \geq 2k$ and $L = [0, k-1]$, or $n < 2k$ and $[2k - n, k-1] \subseteq L$, and $\mathcal{A} = \overline{\mathcal{B}}$;
        \item $k = 2$, $L = \{1,2\}$, and $\mathcal{A} = \mathcal{B} = \left\{ F \in \binom{[n]}{2} : 1 \in F \right\}$;
        \item $k = n - 2$, $L \cap [k-2, k] = \{k-1, k\}$, and $\mathcal{A} = \mathcal{B} = \binom{[n-1]}{n-2}$.
    \end{enumerate}

    \noindent \textbf{\rm(iii)} If $n = 2k$ and $L = k-L$, then
    \[
    |\mathcal{A}| + |\mathcal{B}| \leq \sum_{i \in L} \binom{k}{i} \binom{n-k}{k-i} + 2,
    \]
    and equality holds if and only if (up to isomorphism) one of the following holds:
    \begin{enumerate}
        \item $\mathcal{A} = \{[k], [k+1, 2k]\}$, $\mathcal{B} = \left\{ F \in \binom{[n]}{k} : |F \cap [k]| \in L \right\}$;
        \item $L = [k-1]$, $\emptyset \neq \mathcal{A} = \overline{\mathcal{B}} \subseteq \binom{[n]}{k}$, and $\mathcal{A}$ is closed under complementation, i.e., if $A \in \mathcal{A}$ then $[n] \setminus A \in \mathcal{A}$.
    \end{enumerate}
\end{theorem}

We now turn to the pairwise cross $L$-intersecting case. Shi, Frankl, and Qian \cite{shi2022} determined the maximum total size of pairwise $L$-intersecting families when $L = [k]$, i.e., when every two sets from different families have non-empty intersection. When $L = [k]$, we call the pairwise cross $L$-intersecting families the {\em pairwise cross-intersecting families}.

\begin{theorem}[\cite{shi2022}]
    Let $n, k, t$ be positive integers with $n \geq 2k$ and $r \geq 2$. If non-empty families $\mathcal{A}_1, \mathcal{A}_2, \dots, \mathcal{A}_r \subseteq \binom{[n]}{k}$ are pairwise cross-intersecting, then
    \[
    \sum_{i=1}^r |\mathcal{A}_i| \leq \max\left\{ \binom{n}{k} - \binom{n-k}{k} + r - 1,\; r \binom{n-1}{k-1} \right\}.
    \]
\end{theorem}

Subsequently, Li and Zhang determined the maximum total size of pairwise cross $t$-intersecting families.

\begin{theorem}[\cite{LI2025105960}]\label{t-intersecting}
    Let $n, k, t, r$ be positive integers with $n \geq 2k - t + 1$, $k > t \geq 1$, and $r \geq 2$. If non-empty families $\mathcal{A}_1, \dots, \mathcal{A}_r \subseteq \binom{[n]}{k}$ are pairwise cross $t$-intersecting, then
    \[
    \sum_{i=1}^r |\mathcal{A}_i| \leq \max\left\{ \binom{n}{k} - \sum_{i=0}^{t-1} \binom{k}{i} \binom{n-k}{k-i} + r - 1,\; r \cdot M(n,k,t) \right\},
    \]
    where $M(n,k,t)$ denotes the maximum size of a $t$-intersecting family in $\binom{[n]}{k}$.
\end{theorem}

Many related results can be found in \cite{ZHANG2024103968, HUANG2025105981}. We generalize the above result to pairwise cross $L$-intersecting families.

\begin{theorem}\label{thm: r cross L intersecting} Let $n, k,r$ be positive integers with $ k,  r\ge 2$, and $\emptyset \neq L \subseteq [0,k]$.
    Let $\mathcal{A}_1, \dots, \mathcal{A}_r \subseteq \binom{[n]}{k}$ be non-empty pairwise cross $L$-intersecting families with $|\mathcal{A}_1| \leq |\mathcal{A}_2| \leq \dots \leq |\mathcal{A}_r|$. For sufficiently large $n$, the following results hold.

    \noindent \textbf{\rm(i)} If $L = [0,k]$, then
    \[
    \sum_{i=1}^r |\mathcal{A}_i| \leq r \binom{n}{k},
    \]
    and equality holds if and only if $\mathcal{A}_1 = \dots = \mathcal{A}_r = \binom{[n]}{k}$ (up to isomorphism).

    \noindent \textbf{\rm(ii)} If $L = [t, k]$ for some $t \in [k]$, then
    \[
    \sum_{i=1}^r |\mathcal{A}_i| \leq \max\left\{ \binom{n}{k} - \sum_{i=0}^{t-1} \binom{k}{i} \binom{n-k}{k-i} + r - 1,\; r \binom{n-t}{k-t} \right\}.
    \]

    \noindent \textbf{\rm(iii)} If $L = [0, k-1]$, then
    \[
    \sum_{i=1}^r |\mathcal{A}_i| \leq \binom{n}{k},
    \]
    and equality holds if and only if $\mathcal{A}_1, \dots, \mathcal{A}_r$ are disjoint and $\bigcup_{i=1}^r \mathcal{A}_i = \binom{[n]}{k}$.

    \noindent \textbf{\rm(iv)} If $k \in L$ and $L \neq [t, k]$ for any $t \in [0,k]$, then
    \[
    \sum_{i=1}^r |\mathcal{A}_i| \leq \sum_{i \in L} \binom{k}{i} \binom{n-k}{k-i} + r - 1,
    \]
    and equality holds if and only if $\mathcal{A}_1 = \dots = \mathcal{A}_{r-1} = \{[k]\}$ and $\mathcal{A}_r = \left\{ A \in \binom{[n]}{k} : |A \cap [k]| \in L \right\}$.
\end{theorem}

There are several partial results concerning the maximum sizes of $r$-cross $t$-intersecting families, starting with Hilton \cite{hilton1977intersection} and with Hilton and Milner \cite{hilton1967some}. Recently, Gupta, Mogge, Piga and Sch{\"u}lke \cite{gupta2023r} gave the maximum sum of sizes of $r$-cross $t$-intersecting families.
\begin{theorem}[\cite{gupta2023r}]\label{1}
    Let $r\geq 2$, and $n,t\geq 1$ be integers, $k\in [n]$, and for $i\in [r]$ let $\mathcal{A}_i\subseteq\binom{[n]}{k}$. If $\mathcal{A}_1,\dots,\mathcal{A}_r$ are non-empty $r$-cross $t$-intersecting families and $n\geq 2k-t$, then
    $$\sum_{i=1}^r|\mathcal{A}_i|\leq \max_{m\in [t,k]}\left\{ \sum_{i=t}^k\binom{m}{i}\binom{n-m}{k-i}+(r-1)\binom{n-m}{k-m} \right\},$$
    and this bound is attained.
\end{theorem}
Let $L=[\ell,t]$.
We generalize the result about $r$-cross $t$-intersecting families to $r$-cross $L$-intersecting families. By Theorem \ref{1}, we assume $t\not=k$ and obtain the following result.
\begin{theorem}\label{thm: cross interval intersecting} Let $n, k,r$ be positive integers with $ k,r \geq 2$.
   Let  $\mathcal{A}_1,\dots,\mathcal{A}_r\subseteq\binom{[n]}{k}$ be non-empty $r$-cross $[\ell,s-1]$-intersecting families, where $\ell\in [0,k]$ and $s\in[k]$.  When $n$ is large enough, we have
    $$\sum_{i=1}^r|\mathcal{A}_i|\leq(r-1)\binom{n-\ell}{k-\ell}-\sum_{i=s-\ell}^{k-\ell}\binom{k-\ell}{i}\binom{n-k}{k-\ell-i}+1,$$
    and this bound is attained.
\end{theorem}

For a family $\mathcal{F} \subset \binom{[n]}{k}$ and $0 \leq i \leq n$, the \emph{$i$-shadow} $\partial_i(\mathcal{F})$ is defined by
\(
\partial_i(\mathcal{F}) = \left\{G \in \binom{[n]}{i} : G \subset F \text{ for some } F \in \mathcal{F}\right\}.
\)
If $\mathcal{F}=\{F\}$, we simply write $\partial_i(F)$ for $\partial_i(\mathcal{F})$. For given positive integers $m, k, i$ with $k>i$, the problem of determining the minimum size of $\partial_i(\mathcal{F})$ over all families $\mathcal{F} \subset \binom{[n]}{k}$ of size $m$ was completely solved by Kruskal~\cite{kruskal} and Katona~\cite{Katona}. This result is of fundamental importance in combinatorics and will be useful in our proofs. In fact, we only require a slightly weaker version due to Lov\'{a}sz~\cite{lovasz1979combinatorial}. For any real number $x > k-1$, define $\binom{x}{k} = x(x-1)\cdots(x-k+1)/k!$. Thus, for every positive integer $m$ and fixed $k$, there exists a unique $x > k-1$ such that $\binom{x}{k} = m$.

\begin{theorem}[\cite{lovasz1979combinatorial}]\label{shadow}
Let $1 \leq i < k \leq n$ and $\mathcal{F} \subset \binom{[n]}{k}$ with $|\mathcal{F}| \ge \binom{x}{k}$. Then
$$
|\partial_{i}(\mathcal{F})| \geq \binom{x}{i}.
$$
\end{theorem}

The above theorem establishes that the size of a family $\mathcal{F}$ yields a lower bound on the size of its $i$-shadow. Conversely, the size of the $i$-shadow of $\mathcal{F}$ implies an upper bound on the size of $\mathcal{F}$ itself. This leads naturally to the following corollary.

\begin{cor}\label{cor}
    Let $1 \leq i < k \leq n$ and $\mathcal{F} \subset \binom{[n]}{k}$ with $|\partial_i\mathcal{F}| \leq \binom{x}{i}$. Then
$$
|\mathcal{F}| \leq \binom{x}{k}.
$$
\end{cor}
\begin{proof}
    Suppose $|\mathcal{F}| > \binom{x}{k}$, say $|\mathcal{F}| = \binom{x_0}{k}$ with $x_0>x$. By Theorem \ref{shadow}, $
|\partial_{i}(\mathcal{F})| \geq \binom{x_0}{i}>\binom{x}{i}
$,  a contradiction.
\end{proof}
This paper is organised as follows. In Section 2, we will prove Theorem \ref{thm: 2 cross L intersecting}. In Section 3, we will prove Theorem \ref{thm: r cross L intersecting}, and in Section 4, we will give the proof of Theorem \ref{thm: cross interval intersecting}. Some open problems are given in Section 5.
\section{Cross $L$-intersecting families}
Let $X$ be a finite set and $\Gamma$ be a group transitively acting on $X$. We say that the action of $\Gamma$ on $X$ is \textbf{primitive}, if $\Gamma$ preserves no nontrivial partition of $X$. In any other case, the action of $\Gamma$ is \textbf{imprimitive}. It is easy to see that if the action of $\Gamma$ on $X$ is transitive and imprimitive, then there is a subset $B$ of $X$ such that $1<|B|<|X|$ and $\gamma(B)\cap B=B$ or $\emptyset$ for every $\gamma\in \Gamma$. In this case, $B$ is called an \textbf{imprimitive set} in $X$. Otherwise, $B$ is called a \textbf{primitive set} in $X$. Furthermore, a subset $B$ of $X$ is said to be \textbf{semi-imprimitive} if $1 < |B| < |X|$ and $|\gamma(B) \cap B| = 0$, $1$ or $|B|$ for each $\gamma \in \Gamma$.

Let $G(X,Y)$ be a non-complete bipartite graph and let $A\cup B$ be a nontrivial independent set of $G(X,Y)$ (that is, $A\cup B\not\subseteq X$ and $A\cup B\not\subseteq Y$), where $A\subseteq X$ and $B\subseteq Y$. If $A\subseteq X\setminus N(B)$ and $B\subseteq Y\setminus N(A)$, we have
$$|A|+|B|\leq \max\left\{|A|+|Y|-|N(A)|,|B|+|X|-|N(B)|\right\}.$$ We set $$\epsilon(X)=\min\left\{|N(A)|-|A|:~A\subseteq X \text{~and $N(A)\neq Y$}\right\},$$
and
$$\epsilon(Y)=\min\{|N(B)|-|B|:~B\subseteq Y \text{~and $N(B)\neq X$}\}.$$
Let $\alpha(X,Y)$ denote the  independent number of $G(X,Y)$. Then one sees that
$$\alpha(X,Y)=\max\{|Y|-\epsilon(X),|X|-\epsilon(Y)\}.$$
A subset $A$ of $X$ is called a \textbf{fragment} of $G(X,Y)$ in $X$ if $N(A)\neq Y$ and $|N(A)|-|A|=\epsilon(X).$ We denote $\mathcal{F}(X)$ as the set of all fragments in $X$. Similarly we can define $\mathcal{F}(Y)$. Let $\mathcal{F}(X,Y)=\mathcal{F}(X)\cup \mathcal{F}(Y)$. $ A \in \mathcal{F}(X,Y)$ is  called a $k$-fragment if $|A| = k$.
 A bipartite graph $G(X,Y)$ is called \textbf{part-transitive} if there is a group $\Gamma$ that transitively acts on $X$ and $Y$ respectively and preserves the adjacent relationship. If $G(X,Y)$ is part-transitive, then every two vertices $u_1,u_2\in X$ (resp. $v_1,v_2\in Y$) has the same degree and we use $d(X)$ (resp. $d(Y)$) to denote the degree of vertices in $X$ (resp. $Y$).

We need the following result about the  part-transitive bipartite graphs which proved in \cite{WANG2013129}.
\begin{theorem}[\cite{WANG2013129}]\label{thm: bipartite}
Let $ G(X, Y) $ be a non-complete bipartite graph with $ |X| \leq |Y| $. If $ G(X, Y) $ is part-transitive and every fragment of $ G(X, Y) $ is primitive under the action of a group $ \Gamma $, then
$$ \alpha(X, Y) = |Y| - d(X) + 1. $$
Moreover,
\begin{enumerate}
    \item if $ |X| < |Y| $, then $ X $ has only 1-fragments;
    \item if $ |X| = |Y| $, then each fragment in $ X $ has size 1 or $ |X| - d(X) $ unless there is a semi-imprimitive fragment in $ X $ or $ Y $.
\end{enumerate}
\end{theorem}
Theorem \ref{thm: bipartite} completely solves the case where every fragment of $ G(X, Y) $ is primitive under the action of the group $ \Gamma $. However, we must also consider situations where some fragments are imprimitive.

A bijection $\phi: \mathcal{F}(X, Y) \to \mathcal{F}(X, Y)$ is defined by
$$
\phi(A) =
\begin{cases}
Y \setminus N(A) & \text{if } A \in \mathcal{F}(X), \\
X \setminus N(A) & \text{if } A \in \mathcal{F}(Y).
\end{cases}
$$
The following properties are proposed in \cite{WANG2013129}. The mapping $\phi$ is an involution, satisfying $\phi(\phi(A)) = A$ for all $A \in \mathcal{F}(X, Y)$. Furthermore, it satisfies the size relation $|A| + |\phi(A)| = \alpha(X, Y)$. A fragment $A$ is called \textbf{balanced} if $|A| = |\phi(A)|$. It follows immediately that all balanced fragments have size $\frac{1}{2}\alpha(X, Y)$.
\begin{lemma}[\cite{WANG2013129}]\label{lem: bipartite}
    Let $G(X, Y)$ be a non-complete and part-transitive bipartite graph under the action of a group $\Gamma$. Suppose that $A \in \mathcal{F}(X, Y)$ such that $\emptyset \neq \gamma(A)\cap A \neq A$ for some $\gamma \in \Gamma$. If $|A| \leq |\phi(A)|$, then $A \cup \gamma(A)$ and $A \cap \gamma(A)$ are both in $\mathcal{F}(X, Y)$.
\end{lemma}

 Define $\mathcal{I}(X)$ {\rm(} resp. $\mathcal{I}(Y)${\rm)} to be the collection of all the imprimitive fragments in $X$ {\rm(}resp. $Y${\rm)}.
Then we have the following result.
\begin{theorem}\label{thm: bipartite 2}
Let $ G(X, Y) $ be a non-complete bipartite graph with $ |X| = |Y| $. If $ G(X, Y) $ is part-transitive  under the action of a group $\Gamma$ and $\alpha(X, Y) > |Y| - d(X) + 1$, then there exists an imprimitive fragment in $X$.
Moreover, each fragment in $ X $ of size no more than $\frac{1}{2}\alpha(X, Y)$ is in $\mathcal{I}(X)$ unless there is a primitive fragment $A$ of size no more than $\frac{1}{2}\alpha(X, Y)$ in $ X $ such that $A\cap\gamma(A)\in \mathcal{I}(X)\cup\{\emptyset\}\cup \{A\}$ for any $\gamma\in\Gamma$.
\end{theorem}
\begin{proof} By Theorem \ref{thm: bipartite} and $\alpha(X, Y)>  |Y| - d(X) + 1$,    there exists an imprimitive fragment in $X$.

    %If all the fragments of size no more that $\frac{1}{2}\alpha(X, Y)$ in $X$ are imprimitive. The conclusion holds.

     Assume  there exists a fragment, say $A$, in $X$ such  that $|A|<\frac{1}{2}\alpha(X, Y)$ and $A\notin \mathcal{I}(X)$. We choose  $A$ such that $|A|$ as small as possible.
    Since $\alpha(X, Y) > |Y| - d(X) + 1$, $|A|\ge 2$.  Since $A\notin \mathcal{I}(X)$, $A$ is primitive. For  any $\gamma\in\Gamma$ with $\emptyset \neq \gamma(A)\cap A \neq A$, since $|A|<\frac{1}{2}\alpha(X, Y)$, we have $|A|\le |\phi(A)|$.
  By   Lemma \ref{lem: bipartite}, $\gamma(A)\cap A$ is a fragment. From the minimality of $A$, $\gamma(A)\cap A\in \mathcal{I}(X)$, which ends the proof.
\end{proof}

To determine the extremal cases, we also need the following properties.
\begin{lemma}[\cite{WANG2013129}]\label{lem: bipartite 2}
    Let $G(X, Y)$ be a non-complete bipartite graph with $|X| = |Y|$ and $\epsilon(X) = d(X) - 1$, and let $\Gamma$ be a group acting part-transitively on $G(X, Y)$. If there exists a 2-fragment in $X$, then either
\begin{enumerate}
    \item there exists an imprimitive set $A \subseteq X$ satisfying $|N(A)| - |A| = d(X) - 1$, or

    \item there exists a subset $A \subseteq X$ {\rm(}where either $A = X$ or $A$ is $\Gamma$-imprimitive with $|A| > 2${\rm)} such that the quotient group $\Gamma_A / \left(\bigcap_{a \in A} \Gamma_a\right)$ is isomorphic to a subgroup of the dihedral group $D_{|A|}$. Here, $\Gamma_A = \{\sigma \in \Gamma \mid \sigma(A) = A\}$ denotes the stabilizer of $A$.
\end{enumerate}
\end{lemma}

\begin{lemma}[\cite{WANG2013129}]\label{lem: bipartite 3}
Let $G(X, Y)$ satisfy the hypotheses of Lemma \ref{lem: bipartite 2}. If every fragment of $G(X, Y)$ is primitive and $\mathcal{F}(X, Y)$ contains no 2-fragments, then
    every nontrivial fragment $A \in \mathcal{F}(X)$ (if any exists) is balanced, and
     for each $a \in A$, there exists a unique nontrivial fragment $B$ such that $A \cap B = \{a\}$.
\end{lemma}
%We also need two lemmas in \cite{WANG2013129}.
%\begin{lemma}[\cite{WANG2013129}]\label{lem: old paper lem 1} Let $G(X,Y)$ be a non-partite graph. Then, $|Y|-\epsilon(X)=|X|-\epsilon(Y)$, and(i) $A\in \mathcal{F}(X)$ if and only if $Y\setminus N(A)\in \mathcal{F}(Y)$, and $N(Y\setminus N(A))=X\setminus A$;(ii) $A\cap B$ and $A\cup B$ are both in $\mathcal{F}(X)$ is $A,B\in \mathcal{F}(X)$, $A\cap B$ and $N(A\cup B)\neq Y$.\end{lemma}The above lemma shows that there is a one to one correspondence $\phi: \mathcal{F}(X,Y)\to \mathcal{F}(X,Y)$, where$$ \phi(A)=\left\{\begin{aligned}Y\setminus N(A) & \text{if $A\in \mathcal{F}(X)$} \\X\setminus N(A) & \text{if $A\in \mathcal{F}(Y)$} \end{aligned}\right.$$Moreover, we have $\phi^{-1}=\phi$, and $|A|+|\phi(A)|=\alpha(X,Y)$ if $A\in \mathcal{F}(X,Y)$.\begin{lemma}[\cite{WANG2013129}]\label{lem: old paper lem 2}Let $G(X,Y)$ be a non-complete and part-transitive bipartite graph under the action of a graph $\Gamma$. Suppose that $A\in \mathcal{F}(X,Y)$ such that $\emptyset\neq \gamma(A)\cap A\neq A$ for some $\gamma\in \Gamma$. if $|A|\leq |\phi(A)|$, then $A\cup \gamma(A)$ and $A\cap\gamma(A)$ are both in $\mathcal{F}(X,Y)$.\end{lemma}The above lemma follows that if every element of $X$ ($\mathcal{F}(Y)$) is primitive and there is an $A\in \mathcal{F}(X)$ ($\mathcal{F}(Y)$) with $|A|\leq \phi(A)$, then $\mathcal{F}(X)$ ($\mathcal{F}(Y)$) contains aa singleton. In particular, if $|X|=|Y|$, there are always two kinds of fragments in $X$
Now we are going to prove Theorem \ref{thm: 2 cross L intersecting}. The main idea comes from \cite{WANG2013129}.
\vskip.2cm

\noindent\text{\em Proof of Theorem \ref{thm: 2 cross L intersecting}.}
Let $\mathcal{A},\mathcal{B}\subseteq \binom{[n]}{k}$ be non-empty  cross $L$-intersecting.

We construct  a bipartite graph $G(X,Y)$, where  $X = Y = \binom{[n]}{k}$ and for $A \in X$ and $B \in Y$,  $AB \in E(G)$ if and only if $|A \cap B| \notin L$, meaning $A$ and $B$ are not $L$-intersecting.

Let $S_n$ denote the symmetric group on $[n]$. The group $S_n$ acts transitively on both $X$ and $Y$ while preserving the cross $L$-intersecting property. Consequently,   for any $A \in X$, $ |N_G(A)|=d(X)$ and similarly $|N_G(B)|=d(Y)$ for all $B \in Y$. Let $\overline{L} = \{0,1,\dots,k\} \setminus L$. For any $A \in X$ and $B\in Y$,
$$N_G(A) = \bigcup_{i \in \overline{L}} \left\{B \in \binom{[n]}{k} : |A \cap B| = i\right\}\mbox{~~and~~} N_G(B) = \bigcup_{i \in \overline{L}} \left\{A \in \binom{[n]}{k} : |B \cap A| = i\right\}.$$
  Thus we have
$d(X) = d(Y) = \sum_{i \in \overline{L}} \binom{k}{i}\binom{n-k}{k-i}.$ Moreover, $ |\mathcal{A}| + |\mathcal{B}|\le \alpha(X,Y)$, and the equality holds if and only if   $\mathcal{A}\in\mathcal{F}(X)$ and $\mathcal{B}=Y\setminus N(\mathcal{A})$. Consequently, we transform the original problem into determining $\alpha(X,Y)$ and $\mathcal{F}(X,Y)$.

It is known that for $n \neq 2k$, the action of $S_n$ on $\binom{[n]}{k}$ is primitive \cite{WANG2011455}, making every fragment of $G(X,Y)$ primitive. For $n \geq 2k+1$ with $L \neq [0,k]$, or when $k \leq n < 2k$ and $[2k-n,k] \cap L \notin\{ \emptyset,[2k-n,k]\}$,  $G(X,Y)$ is not  complete bipartite. By Theorem \ref{thm: bipartite}, we obtain
\begin{align}\label{eq1}
    \alpha(X,Y) &= |Y| - d(X) + 1 = \binom{n}{k} - \sum_{i \in \overline{L}} \binom{k}{i}\binom{n-k}{k-i} + 1\nonumber\\
    &=\sum_{i \in L} \binom{k}{i}\binom{n-k}{k-i} + 1.
\end{align}

Now, we handle the case $n = 2k$. According to \cite{WANG2011455}, the only imprimitive sets are complementary pairs $\{A, \overline{A}\}$ where $A \in \binom{[2k]}{k}$ and $\overline{A} = [n] \setminus A$. Let $A,\overline{A}\in X$. When $L = k-L$, for any such pair we have
$$\alpha(X,Y) \geq |Y| - |N(\{A, \overline{A}\})| + 2 = |Y| - d(X) + 2 > |Y| - d(X) + 1.$$
By Theorem \ref{thm: bipartite 2}, $\{A, \overline{A}\}$ forms an imprimitive fragment for some $A \in \binom{[n]}{k}$. So we have
\begin{align}\label{eq2}
\alpha(X,Y) = |Y| - d(X) + 2 = \sum_{i \in L} \binom{k}{i}\binom{n-k}{k-i} + 2.
\end{align}
When $L \neq k-L$, we have that $|Y| - |N(\{A, \overline{A}\})| + 2 < |Y| - d(X) + 1$. So $\{A, \overline{A}\}$ cannot be a fragment which implies every fragment of $ G(X, Y) $ is primitive under the action of  $ S_n $. By Theorem \ref{thm: bipartite},
\begin{align}\label{eq3}
    \alpha(X,Y) = |Y| - d(X) + 1
    =\sum_{i \in L} \binom{k}{i}\binom{n-k}{k-i} + 1.
\end{align}

We complete the proof by considering three cases.
\begin{mycase}

\case $n\geq2k,L=[0,k]$ or $n<2k,[2k-n,k]\subseteq L$.

In these cases,  $G(X,Y)$ is an empty graph. Then $|\mathcal{A}| + |\mathcal{B}|\leq \alpha(X,Y)=2\binom{n}{k}$, with equality when $\mathcal{A} = \mathcal{B} = \binom{[n]}{k}$.

\case $n>2k,L\neq [0,k]$ or $n=2k,L\neq k-L$ or $n<2k,[2k-n,k]\cap L\notin\{ \emptyset,[2n-k,k]\}$.

In these cases, $|\mathcal{A}| + |\mathcal{B}|\leq \alpha(X,Y)=\sum_{i \in L} \binom{k}{i}\binom{n-k}{k-i} + 1$ by (\ref{eq1}) and (\ref{eq3}), and the equality holds if and only if   $\mathcal{A}\in\mathcal{F}(X)$ and $\mathcal{B}=Y\setminus N(\mathcal{A})$. By Theorem \ref{thm: bipartite},  the equality holds if $|\mathcal{A}|=1$ or $|\mathcal{B}|=1$ unless there is a semi-imprimitive fragment in $\mathcal{F}(X)$.
We now proceed to characterize all extremal cases by determining all semi-imprimitive fragment in $\mathcal{F}(X)$.

If $|\mathcal{A}|=1$ or $|\mathcal{B}|=1$, say $|\mathcal{A}|=1$, then we let $\mathcal{A}=\{[k]\},\mathcal{B}=\{F\in \binom{[n]}{k}:|F\cap [1,k]|\in L\}$ and we are done.

If $n\geq2k,L=[0,k-1]$ or $n< 2k,L\cap [2k-n,k]=[2k-n,k-1]$, then for any $\emptyset \neq S\subsetneqq \binom{[n]}{k}$, we have $S\in \mathcal{F}(X,Y)$. Then we let $\mathcal{A}=\overline{\mathcal{B}}$ and we are done.

Now, we assume $n\geq2k,L\neq[0,k-1]$ or $n< 2k,L\cap [2k-n,k]\neq[2k-n,k-1]$, and assume that $S\in\mathcal{F}(X)$ is semi-imprimitive. Then $2\le |S|< \binom{n}{k}$. We first have the following claim.
\begin{claim}\label{Claim fra1}
    There is a $k$-set $C\in S$, such that either $|C \cap B| = 1$ for every $B \in S\setminus\{C\}$, or $|C \cap B| = k - 1$ for every $B \in S\setminus\{C\}$.
\end{claim}

\noindent{\bf Proof of Claim \ref{Claim fra1}}
All fragments in this case are primitive, and $S_n$ is not isomorphic to a subgroup of $D_{|X|}$ for $n \geq 4$. By Lemma \ref{lem: bipartite 2}, there are no $2$-fragments in $G(X,Y)$, which implies that $S$ is balanced by Lemma \ref{lem: bipartite 3}.

For each subset $C \subseteq [n]$, let $S_C=\{\sigma \in S_n :  \sigma \mbox{~fix all elements of~} \overline{C}\}$. Given  $C \in S$, let $\Gamma^C = S_C \times S_{\overline{C}}$, $\Gamma^C_S = \{\sigma \in \Gamma^C : \sigma(S) = S\}$, $\Gamma_S = \{\sigma \in S_n : \sigma(S) = S\}$. Then $C \in \sigma(S)$ for each $\sigma \in \Gamma^C_S$. Since $|S|\ge 2$, there exists some $C\in S$ for which $\Gamma^C_S \neq \Gamma^C$. (Otherwise, we have $S_C \times S_{\overline{C}}=\Gamma^C_S\subseteq \Gamma_S, \Gamma^B=S_B \times S_{\overline{B}}=\Gamma^B_S\subseteq \Gamma_S$ for any $B \in S \backslash \{C\}$. Since a complementary pair is not a fragment by directly calculation, assume $B\neq \overline{C}$. Then, $S_C \times S_{\overline{C}}$ and $S_B \times S_{\overline{B}}$ will generate $S_n$, which implies $\Gamma_S=S_n$ and $S=\binom{[n]}{k}$, a contradiction with $S$ being a fragment.) Lemma \ref{lem: bipartite 3} then shows that $[\Gamma^C : \Gamma^C_S]$, the index of $\Gamma^C_S$ in $\Gamma^C$, equals 2.

Let $\Gamma_S [C]$ denote the projection of $\Gamma^C_S$ onto $S_C$. Then $\Gamma_S [C]$ is a subgroup of $S_C$ with index $\leq 2$, meaning $\Gamma_S [C] = S_C$ or $A_C$, where $A_C$ is an alternating group on $C$. Consequently, the only index-2 subgroups of $\Gamma^C$ are $A_C \times S_{\overline{C}}$ and $S_C \times A_{\overline{C}}$. So $\Gamma^C_S$ must be one of them.

For each $B \in S \backslash \{C\}$, we have $k = |B \cap C| + |B \cap \overline{C}|$. If $|B \cap C| > 1$ (resp. $|B \cap \overline{C}|>1$), let $(i,j)$ be a transposition with $i, j \in B \cap C$ (resp. $i, j \in B \cap \overline{C}$), then $(i,j)$ fixes both $B$ and $C$ (resp. $\overline{C}$). By the semi-imprimitivity of $S$, we have $(i,j) \in \Gamma^C_S$, which forces $\Gamma^C_S = S_C \times A_{\overline{C}}$ (resp. $\Gamma^C_S =A_C \times S_{\overline{C}}$). This process reveals that for each $B \in S\setminus\{C\}$, at most one of $|B \cap C|$ and $|B \cap \overline{C}|$ exceeds 1.

If $B \subseteq \overline{C}$, then both $S_C$ and $S_B$ fix $C$ and $B$, implying $S_C \times S_B \subseteq \Gamma^C_S$. However, neither $A_C \times S_{\overline{C}}$ nor $S_C \times A_{\overline{C}}$ contains $S_C \times S_B$, a contradiction. We thus conclude that either $|C \cap B| = 1$ for every $B \in S\setminus\{C\}$, or $|C \cap B| = k - 1$ for every $B \in S\setminus\{C\}$.
\q

We now characterize all nontrivial fragments of $G(X,Y)$ and the corresponding set $L$. By Claim \ref{Claim fra1}, let $C\in S$ such that either $|C \cap B| = 1$ for every $B \in S\setminus\{C\}$, or $|C \cap B| = k - 1$ for every $B \in S\setminus\{C\}$.

Assume $|C \cap B| = 1$ for every $B \in S\setminus\{C\}$. Without loss of generality, assume $C \cap B = \{1\}$ for some $B \in S$. If $k > 2$, then $|B \cap \overline{C}| \geq 2$, so $\Gamma^C_S = A_C \times S_{\overline{C}}$. However, we can find distinct $i, j \in C\setminus\{1\}$ such that $(1 i j)(B) = (B \setminus \{1\}) \cup \{i\} \in S$ since $(1 i j) \in A_C$. Consequently, $(1 i)(S)$ contains more than one element of $S$, so $(1 i) \in \Gamma^C_S$, leading to a contradiction. This proves $k = 2$, meaning $S$ consists of all $2$-subsets $\{1, i\}$ for $i \in [2, n]$. Since $L \subseteq [0,k]$, $L$ must be one of $\{0\},\{1\},\{2\},\{0,1\},\{0,2\},\{1,2\}$. Computing $d(X)$ and $N_{G}(S)$ while satisfying $|N_{G}(S)| - |S| = d(X) - 1$ reveals that $L=\{1,2\}$. In this case, we let $\mathcal{A}=\mathcal{B}=\{F\in\binom{[n]}{2}:1\in F\}$ and we are done.

Next, assume $|C \cap B| = k - 1 > 1$ for every $B \in S\setminus\{C\}$. Similarly, we can show that $n - k = 2$, $L\cap [k-2,k]= \{k-1,k\}$, and $\Gamma^C_S =  S_C \times A_{\overline{C}}$. Thus, $S = \binom{A}{n-2}$, where $A$ is an $(n-1)$-subset of $[n]$. Therefore, $\mathcal{A}=\mathcal{B}=\binom{A}{n-2}$ is the extremal structure, where $A$ is an $(n-1)$-subset of $[n]$.

\case$n=2k$, $L=k-L\neq [0,k]$.

In this case, $|\mathcal{A}| + |\mathcal{B}|\leq \alpha(X,Y)=\sum_{i \in L} \binom{k}{i}\binom{n-k}{k-i} + 2$ by (\ref{eq2}). If $\mathcal{A}=\{A,\overline{A}\}$ or $\mathcal{B}=\{B,\overline{B}\}$ (i.e.,  a complementary pair), the equality holds by the discussion above, where $A,B\in \binom{[n]}{k}$. Thus complementary pairs are in $\mathcal{F}(X)=\mathcal{F}(Y)$. Let $S\in \mathcal{F}(X)$. Then $|S|\le \frac{1}{2}\alpha(X,Y)$.

Since $\alpha(X,Y)=\sum_{i\in L}\binom{k}{i}\binom{n-k}{k-i}+2$, by the discussion above, $|S|\ge 2$. If $|S|=2$ for any $S\in \mathcal{F}(X)$, then $S$ is  a complementary pair. Let $\mathcal{A}=\{[k],[k+1,2k]\},\mathcal{B}=\{F\in \binom{[n]}{k}:|F\cap [k]|\in L\}$ and we are done.

Now we assume there is $S\in \mathcal{F}(X)$ such that $|\phi(S)|\geq|S|\ge3$. Choose $S$ such that $|S|$ as smaller as possible.
 Note that for any $\gamma\in S_n$, $\gamma(S)\cap S$ is either empty, equal to $S$, or a complementary pair (otherwise, $\gamma(S)\cap S$ would be a fragment smaller than $S$ by Lemma \ref{lem: bipartite}, a contradiction with the choice of $S$). The following properties hold for $S$.

\begin{claim}\label{claim fra2}
\rm{(i)} For any $A\in S$, $\overline{A}\in S$.

\noindent\rm{(ii)} $S$ consists of 2 or 3 different pairs of complementary $k$-sets. Moreover, if $S$ consists of 3 such pairs, it must have the form $\{A_1\cup A_2,A_1\cup A_3,A_1\cup A_4,A_2\cup A_3,A_2\cup A_4,A_3\cup A_4\}$, where $\{A_i\}^4_{i=1}$ is a balanced partition of $[2k]$ with $|A_i|=\frac{k}{2}$.
\end{claim}

\noindent{\bf Proof of Claim \ref{claim fra2}}
    \rm{(i)}
     Suppose there exist a complementary pair $\{B,\overline{B}\}\subseteq S$ and a set $A\in S$ but $\overline{A}\notin S$. Choose $\gamma\in S_n$ such that $\gamma(A)=B$. Then $B\in \gamma(S)$, but $\overline{B}\notin \gamma(S)$, meaning $\gamma(S)\cap S$ is neither empty, nor $S$, nor a complementary pair, a contradiction.

    Suppose for any $B\in S$, we have $\overline{B}\notin S$.
   Then $\gamma(S)\cap S$ can never be a complementary pair for any $\gamma\in S_n$ which implies  $S$ is primitive. However, all primitive sets are complementary pairs, a contradiction.

    \noindent\rm{(ii)} By (i) and $|S|\ge3$, $S$ consists of complementary $k$-set pairs.  Assume $S$ contains at least 3  complementary $k$-set pairs. Let $\{A,B,\overline{A},\overline{B}\}\subsetneqq S$ and $\Gamma_S=\{\gamma\in S_n:\gamma(S)=S\}$.

    Note that $\{A,B,\overline{A},\overline{B}\}\subseteq \gamma(S)\cap S$ for any $\gamma\in S_{A\cap B}\cup S_{A\cap \overline{B}}\cup S_{\overline{A}\cap B}\cup S_{\overline{A}\cap \overline{B}}$. Thus $S_{A\cap B}\cup S_{A\cap \overline{B}}\cup S_{\overline{A}\cap B}\cup S_{\overline{A}\cap \overline{B}}\subseteq \Gamma_S$. For any $\{C,\overline{C}\}\subseteq S\setminus\{A,B,\overline{A},\overline{B}\}$, we similarly have $S_{A\cap C}\cup S_{A\cap \overline{C}}\cup S_{\overline{A}\cap C}\cup S_{\overline{A}\cap \overline{C}}\subseteq \Gamma_S$.

    If $\{A\cap C,A\cap\overline{C},\overline{A}\cap C,\overline{A}\cap\overline{C}\}\neq \{A\cap B,A\cap\overline{B},\overline{A}\cap B,\overline{A}\cap\overline{B}\}$, then $S_{A\cap B}$, $S_{A\cap \overline{B}}$, $S_{\overline{A}\cap B}$, $S_{\overline{A}\cap \overline{B}}$, $S_{A\cap C}$, $S_{A\cap \overline{C}}$, $S_{\overline{A}\cap C}$, $S_{\overline{A}\cap \overline{C}}$ generate $S_n$, implying $S=\binom{[2k]}{k}$, a contradiction. Thus, $\{A\cap C,A\cap\overline{C},\overline{A}\cap C,\overline{A}\cap\overline{C}\}= \{A\cap B,A\cap\overline{B},\overline{A}\cap B,\overline{A}\cap\overline{B}\}$, and the only possible configuration is $\{A_1\cup A_2,A_1\cup A_3,A_1\cup A_4,A_2\cup A_3,A_2\cup A_4,A_3\cup A_4\}$, where $\{A_i\}^4_{i=1}$ partitions $[2k]$ with $|A_i|=\frac{k}{2}$.
    \q

 Claim \ref{claim fra2} characterize the structure of $S$, and we now determine all possible $L$. Since $L=k-L$, we have $x\in L$ if and only if $k-x\in L$. We complete the proof by considering two subcases.

 \noindent {\bf Subcase 3.1} $S= \{A,B,\overline{A},\overline{B}\}$.

    Since $S\in \mathcal{F}(X)$ and $N(\{A,\overline{A}\})=N(A)$, we have
    \begin{align}\label{eq}
        d(X)-2&=|N(S)|-|S| \nonumber\\
        &=|N(\{A,\overline{A}\})|+|N(\{B,\overline{B}\})\setminus N(\{A,\overline{A}\})|-4\nonumber\\
        &=d(X)-4+|N(\{B\})\setminus N(\{A\})|.
    \end{align}
    Thus, $|N(\{B\})\setminus N(\{A\})|=2$.

    From $|N(S)|-|S|=|N(\{A,\overline{A}\})|-|\{A,\overline{A}\}|$, we derive $|N(\{B,\overline{B}\})\setminus N(\{A,\overline{A}\})|=2$. Assume $|A\cap B|=a\leq\frac{k}{2}$ (otherwise, replace $B$ by $\overline{B}$). Let $$F_{w,x,y,z}=\left\{F\in \binom{[2k]}{k}:|F\cap A\cap B|=w,|F\cap A\cap \overline{B}|=x,|F\cap \overline{A}\cap B|=y,|F\cap \overline{A}\cap \overline{B}|=z\right\}.$$ So $|F_{w,x,y,z}|=\mathbf{1}_{w+x+y+z=k}(w,x,y,z)\binom{a}{w}\binom{a}{z}\binom{k-a}{x}\binom{k-a}{y}$. For fixed $w,x,y,z$, either $F_{w,x,y,z}\subseteq N(\{B,\overline{B}\})\setminus N(\{A,\overline{A}\})$ or $F_{w,x,y,z}\cap (N(\{B,\overline{B}\})\setminus N(\{A,\overline{A}\}))=\emptyset$. Denote $$g_L(w,x,y,z)=\mathbf{1}_{w+x+y+z=k}(w,x,y,z)\mathbf{1}_{w+x\in L}(w,x)\mathbf{1}_{w+y\notin L}(w,y)\binom{a}{w}\binom{a}{z}\binom{k-a}{x}\binom{k-a}{y}$$ for $0\leq w,x,y,z\leq k$. Then
    \begin{align}
        2=|N(\{B,\overline{B}\})\setminus N(\{A,\overline{A}\})|
        &=\sum_{0\leq w,z\leq a,0\leq x,y\leq k-a}g_L(w,x,y,z).\label{eq5}
    \end{align}

    Suppose $\{0,k\}\subseteq L$. There exists at least one $(w,x,y,z)$ with $g_L(w,x,y,z)>0$. Fix such a tuple. If any of $\binom{a}{w},\binom{a}{z},\binom{k-a}{x},\binom{k-a}{y}$ exceeds 1, by $w+x+y+z=k$, we have either at least two of them do, or $(a-w,k-a-x,k-a-y,a-z)\neq (w,x,y,z)$ with $g(a-w,k-a-x,k-a-y,a-z)=g_L(w,x,y,z)$. But in those two cases, we have  $|N(\{B,\overline{B}\})\setminus N(\{A,\overline{A}\})|\geq \sum_{0\leq w,x,y,z\leq k} g_L(w,x,y,z)\geq 4$, a contradiction with (\ref{eq5}). Thus we have $(w,x,y,z)=(0,0,a,k-a)$ or $(a,k-a,0,0)$, and then  $a,k-a\notin L$. If $\{0,k\}\subsetneqq L$,  there exists $\{b,k-b\}\subseteq L$ with $b\leq\frac{k}{2}$. Then $(w,x,y,z)=(a,k-2a,b-a,2a-b)$ (if $b>a$) or $(a,0,k-b-a,b)$ (if $a>b$) would satisfy $g_L(w,x,y,z)\geq 4$, a contradiction with (\ref{eq5}). Thus, $L=\{0,k\}$ is the only possible solution of (\ref{eq5}).

    Suppose  $\{0,k\}\subseteq\overline{L}$. By (\ref{eq5}), we have
    \begin{align}
        2=|N(\{B,\overline{B}\})\setminus N(\{A,\overline{A}\})|
        &=\sum_{0\leq w,z\leq a,0\leq x,y\leq k-a}g_L(w,x,y,z)\nonumber\\
        &=\sum_{0\leq w,z\leq a,0\leq x,y\leq k-a}g_L(a-w,k-a-x,k-a-y,a-z)\nonumber\\
        &=\sum_{0\leq w,z\leq a,0\leq x,y\leq k-a}g_{\overline{L}}(w,x,y,z).\label{eq6}
    \end{align}
    By (\ref{eq6}) and the same discussion  above, we have derive a contradiction if $\{0,k\}\subsetneqq \overline{L}$. Thus, $\overline{L}=\{0,k\}$ and $L=[k-1]$.

    From the analysis above, $L$ must be $\{0,k\}$ or $[k-1]$. If $L=\{0,k\}$, we have $N(S)=Y$, which implies $S$ is not a fragment, a contradiction. If $L=[k-1]$, easily verify that any nontrivial family $\mathcal{A}$ of $k$-sets closed under complementation is a fragment. Let $\emptyset\neq\mathcal{A}=\overline{\mathcal{B}}\subseteq \binom{[n]}{k}$ and we are done.

\vskip.2cm
     \noindent {\bf Subcase 3.2} $S=\{A_1\cup A_2,A_1\cup A_3,A_1\cup A_4,A_2\cup A_3,A_2\cup A_4,A_3\cup A_4\}$, where $\{A_i\}^4_{i=1}$ partitions $[2k]$ with $|A_i|=\frac{k}{2}$.

     Let $A=A_1\cup A_2$, $B=A_1\cup A_3$, $C=A_1\cup A_4$.
    We have
    \begin{align*}
        d(X)-2&=|N(S)|-|S|\\
        &=|N(\{A,\overline{A}\})|+|N(\{B,\overline{B}\})\setminus N(\{A,\overline{A}\})|\\&+|N(\{C,\overline{C}\})\setminus (N(\{A,\overline{A}\})\cup N(\{B,\overline{B}\}))|-6\\
        &\geq d(X)+|N(\{B\})\setminus N(\{A\})|-6.
    \end{align*}
    Thus, $|N(\{B\})\setminus N(\{A\})|\leq 4$. On the other hand,
    \begin{align*}
         d(X)-2&=|N(S)|-|S|\\
        &\leq |N(\{A\})|+|N(\{B\})\setminus N(\{A\})|+|N(\{C\})\setminus N(\{A\}))|-6.\\
    \end{align*}
    Thus, $|N(\{B,\overline{B}\})\setminus N(\{A,\overline{A}\})|+|N(\{C,\overline{C}\})\setminus N(\{A,\overline{A}\}))|\geq 4$. Without loss of generality, assume $|N(\{B\})\setminus N(\{A\})|>0$.

    Let $$F_{w,x,y,z}=\left\{F\in \binom{[2k]}{k}:|F\cap A_1|=w,|F\cap A_2|=x,|F\cap A_3|=y,|F\cap A_4|=z\right\}.$$ So $|F_{w,x,y,z}|=\mathbf{1}_{w+x+y+z=k}(w,x,y,z)\binom{k/2}{w}\binom{k/2}{z}\binom{k/2}{x}\binom{k/2}{y}$.  Write $$g_L(w,x,y,z)=\mathbf{1}_{w+x+y+z=k}(w,x,y,z)\mathbf{1}_{w+x\in L}(w,x)\mathbf{1}_{w+y\notin L}(w,y)\binom{k/2}{w}\binom{k/2}{z}\binom{k/2}{x}\binom{k/2}{y}$$ for $0\leq w,x,y,z\leq \frac{k}{2}$.
    Therefore
    \begin{align}\label{eq7}
        4\geq|N(\{B\})\setminus N(\{A\})|
       =\sum_{0\leq w,x,y,z\leq k/2}g_L(w,x,y,z).
    \end{align}

    If  $g_L(w,x,y,z)\le1$ for any $0\leq w,x,y,z\leq k/2$, then the only possible solution such that $g_L(w,x,y,z)>0$ is $(w,x,y,z)=(0,0,a,k-a)$ or $(a,k-a,0,0)$. Thus, $|N(\{B,\overline{B}\})\setminus N(\{A,\overline{A}\})|=2$, meaning $\{A,B,\overline{A},\overline{B}\}$ is a fragment, a contradiction to the minimality  of $S$.

    If there is a $(w,x,y,z)$ satisfies $g_L(w,x,y,z)>1$, then $g_L(w,x,y,z)\geq 4$ (since at least two binomial coefficients exceed 1). The equality holds only when $k=4$ and $(w,x,y,z)=(2,1,0,1)$, $(1,1,0,2)$, $(2,0,1,1)$, or $(1,0,1,2)$. In each case, $g(2-w,2-x,2-y,2-z)=4$. So $|N(\{B\})\setminus N(\{A\})|\geq 8$, a contradiction with (\ref{eq7}).
 \hfill$\square$
\end{mycase}

\section{Pairwise cross $L$-intersecting}
For any $S\subseteq [n]$ and $\mathcal{F}\subseteq\binom{[n]}{k}$, define $\mathcal{F}(S)=\{F\in \mathcal{F}:S\subseteq F\}$ and $\hat{\mathcal{F}}(S)=\{F\setminus S:F\in \mathcal{F}(S)\}$. Let $L$ be a set of integers with $L\subseteq[n]$ and $\mathcal{A}_1, \mathcal{A}_2, \dots, \mathcal{A}_r \subseteq \binom{[n]}{k}$ be pairwise non-empty cross-intersecting families.
When $L=[0,k]$, it's trivial to have $\sum_{i=1}^r|\mathcal{A}_i|\leq r\binom{n}{k}$.
Let $L=[0,k-1]$. For every $A\in \mathcal{A}_i$ and $B\in \mathcal{A}_j$ with $i\neq j$, we have $A\neq B$. So $\mathcal{A}_1,\dots,\mathcal{A}_r$ are disjoint and then $\sum_{i=1}^r|\mathcal{A}_i|\leq \binom{n}{k}$. When $L=[t,k]$ for some $t\in[k]$, it has been done by Theorem \ref{t-intersecting}. Then we left the case (iv) in Theorem \ref{thm: r cross L intersecting}. Recall the conditions in Theorem \ref{thm: r cross L intersecting}(iv) are $k\in L$ and $L\neq [t,k]$ for any $t\in [0,k]$.

Let $\mathcal{A}_1,\dots,\mathcal{A}_r\subseteq \binom{[n]}{k}$ be the non-empty pairwise cross $L$-intersecting families, which obtain the maximum $\sum^r_{i=1}|\mathcal{A}_i|$.
Set $\mathcal{B}_1=\mathcal{B}_2=\dots=\mathcal{B}_{r-1}=\{[k]\}$ and $\mathcal{B}_r=\{A\in \binom{[n]}{k}:~|A\cap [k]|\in L\}$. Since $k\in L$, $\mathcal{B}_1,\dots,\mathcal{B}_r$ are pairwise cross $L$-intersecting. Then
  \begin{align}\label{seq4-1}
  \sum^r_{i=1}|\mathcal{A}_i|\ge \sum_{i\in L}\binom{k}{i}\binom{n-k}{k-i}+r-1=\binom{n}{k}-\binom{k}{s}\binom{n-s}{k-s}-o(n^{k-s}).\end{align}

We first give some lemmas which will be used in the proof of our result  by considering whether $0\in L$. First, we deal with the case where $0\in L$.
\subsection{The case when $0\in L$}
In this subsection, let $n,k,r$ be positive integers with $n\geq 2k$ and $k, r\geq 2$. Let $0,k\in L$  and $L\neq [t,k]$ for any $t\in [0,k]$. Let $s$ be the minimum non-negative integer such that $s\notin L$, then $1\leq s\leq k-1$. Assume   $|\mathcal{A}_1|\leq |\mathcal{A}_2|\leq \dots \leq |\mathcal{A}_r|$.
 We have the following results.
\begin{lemma}\label{3.1}
    For any $S\in \binom{[n]}{s}$, one of the following results holds:
    \begin{enumerate}
        \item $\sum^r_{i=1}|\mathcal{A}_i(S)|\leq \max\left\{\binom{n-s}{k-s}-\binom{n-k}{k-s}+r-1,r\binom{n-s-1}{k-s-1}\right\}$.
        \item There are  exactly $r-1$ families in $\mathcal{A}_1(S), \mathcal{A}_2(S), \ldots, \mathcal{A}_r(S)$ are empty.
    \end{enumerate}
\end{lemma}
\begin{proof}
For any $S \in \binom{[n]}{s}$, if all families $\mathcal{A}_1(S), \mathcal{A}_2(S), \ldots, \mathcal{A}_r(S)$ are empty, then  the conclusion holds trivially. %If among $\mathcal{A}_1(S), \mathcal{A}_2(S), \ldots, \mathcal{A}_r(S)$ exactly $r-1$ are empty, the secound condition holds.
Therefore, we only need to consider the case where at least two families of $\mathcal{A}_1(S), \mathcal{A}_2(S), \ldots, \mathcal{A}_r(S)$ are non-empty. Assume $\mathcal{A}_{i_1}(S), \mathcal{A}_{i_2}(S), \ldots, \mathcal{A}_{i_{r'}}(S)$ are the non-empty families, where  $2\le r' \leq r$. %, and their corresponding families $\widehat{\mathcal{A}}_{i_1}(S), \widehat{\mathcal{A}}_{i_2}(S), \ldots, \widehat{\mathcal{A}}_{i_{r'}}(S)$.
 %These families $\mathcal{A}_1(S), \mathcal{A}_2(S), \ldots, \mathcal{A}_r(S)$ and $\widehat{\mathcal{A}}_{i_1}(S), \widehat{\mathcal{A}}_{i_2}(S), \ldots, \widehat{\mathcal{A}}_{i_{r'}}(S)$ satisfy the following fundamental property.
We have the following claim.
\begin{claim}\label{C_1}
The families $\hat{\mathcal{A}}_{i_1}(S), \hat{\mathcal{A}}_{i_2}(S), \ldots, \hat{\mathcal{A}}_{i_{r'}}(S)$ are pairwise intersecting.
\end{claim}

\noindent{\bf Proof of Claim \ref{C_1}}
By contradiction. Suppose, without loss of generality, there are $F_1 \in \hat{\mathcal{A}}_{i_1}(S)$ and $F_2 \in \hat{\mathcal{A}}_{i_2}(S)$ with $F_1 \cap F_2 = \emptyset$. Then $F_1\cup S\in \mathcal{A}_{i_1}$ and $F_2\cup S\in \mathcal{A}_{i_2}$ and $|(F_1\cup S)\cap(F_2\cup S)| = s \notin L$, a contradiction.
\q

To complete the proof of Lemma \ref{3.1}, we proceed by case analysis on $s$.

\textbf{Case 1:} $1\leq s\leq k-2$.

Since $\hat{\mathcal{A}}_{i_1}(S),\ldots,\hat{\mathcal{A}}_{i_{r'}}(S)\subseteq\binom{[n]\setminus S}{k-s}$, by Claim \ref{C_1} and Theorem~\ref{t-intersecting}, we have
 \begin{align*}
\sum^r_{i=1}|\mathcal{A}_i(S)| &= \sum^{r'}_{j=1}|\hat{\mathcal{A}}_{i_j}(S)| \leq \max\left\{\binom{n-s}{k-s}-\binom{n-k}{k-s}+r'-1, r'\binom{n-s-1}{k-s-1}\right\}\\
&\le\max\left\{\binom{n-s}{k-s}-\binom{n-k}{k-s}+r-1, r\binom{n-s-1}{k-s-1}\right\}.
 \end{align*}

\textbf{Case 2:} $s=k-1$.

In this case, the $1$-uniform families $\hat{\mathcal{A}}_{i_j}(S)$ must be identical singletons  by Claim \ref{C_1}, that is, there exists $x\in[n]$ with $\hat{\mathcal{A}}_{i_j} = \{x\}$ for all $j\in[r']$. Thus
$$
\sum^r_{i=1}|\mathcal{A}_i(S)| = r' \leq r = \max\left\{\binom{n-s}{k-s}-\binom{n-k}{k-s}+r-1, r\binom{n-s-1}{k-s-1}\right\},
$$
where the equality holds because $\binom{n-(k-1)}{1}- \binom{n-k}{1} + r - 1 = r-1$ and $r\binom{n-(k-1)-1}{0} = r$.
\end{proof}

For  $\mathcal{F}\subseteq\binom{[n]}{k}$, we define $$\mathcal{S}_s(\mathcal{F})=\left\{S\in \binom{[n]}{s}:|\mathcal{F}(S)|> \frac{3}{2}\binom{n-s}{k-s}-\binom{n-k}{k-s}\right\}.$$
\begin{lemma}\label{3.3}
    The families $\mathcal{S}_s(\mathcal{A}_1),\dots,\mathcal{S}_s(\mathcal{A}_r)\subseteq \binom{[n]}{s}$ are pairwise cross $[0,\max\{0,2s-k-1\}]$-intersecting.
\end{lemma}
\begin{proof}
    To the contrary, without loss of generality, suppose that $|S_1 \cap S_2| > \max\{0,2s-k-1\}$, where $S_1\in \mathcal{S}_s(\mathcal{A}_1)$ and $S_2\in \mathcal{S}_s(\mathcal{A}_2)$. Since $2s-k-1<|S_1 \cap S_2|\leq s$, we have  $|S_1\cup S_2|=2s-|S_1\cap S_2|\leq k$.

    Define $\mathcal{G}_1=\left\{F\in \mathcal{A}_1(S_1):F\cap S_2=S_1\cap S_2\right\}$, $\mathcal{G}_2=\left\{F\in \mathcal{A}_2(S_2):F\cap S_1=S_1\cap S_2\right\}$ and define $\mathcal{G}'_1=\left\{G\setminus S_1:G\in \mathcal{G}_1\right\}$, $\mathcal{G}'_1=\left\{G\setminus S_2:G\in \mathcal{G}_2\right\}$. For  $\{i,j\}=\{1,2\}$, we have
    \begin{align*}
        |\mathcal{G}'_i|&=|\mathcal{A}_i(S_i)|-|\left\{F\in \mathcal{A}_i(S_i):(F\setminus S_i)\cap S_j\neq\emptyset\right\}|\\
        &> \frac{3}{2}\binom{n-s}{k-s}-\binom{n-k}{k-s}-\left(\binom{n-s}{k-s}-\binom{n-k}{k-s}\right)\\
        &=\frac{1}{2}\binom{n-s}{k-s}.
    \end{align*}
    Therefore,
    \begin{align*}
        |\mathcal{G}'_1|+|\mathcal{G}'_2|&\ge \binom{n-s}{k-s}+1.
    \end{align*}
    Noticing that $\mathcal{G}'_1,\mathcal{G}'_2\subseteq \binom{[n]\setminus (S_1\cup S_2)}{k-s}$ and
    $$\binom{n-s}{k-s}>\binom{n-|S_1\cup S_2|}{k-s}-\binom{k-s}{s-|S_1\cap S_2|}\binom{n-|S_1\cup S_2|-k+s}{k-|S_1\cap S_2|},$$
    by Theorem \ref{thm: 2 cross L intersecting} (ii), $\mathcal{G}'_1,\mathcal{G}'_2$ can not be cross ($[0,k-s]\setminus \{s-|S_1\cap S_2|\}$)-intersecting (note that $0\leq s-|S_1\cap S_2|\leq k-s$).
    Then, there exist $G_1\in \mathcal{G}'_1$ and $G_2\in \mathcal{G}'_2$ such that $|G_1\cap G_2|=s-|S_1\cap S_2|$. Therefore, $|(G_1\cup S_1)\cap(G_2\cup S_2)|=s\notin L$, where $G_1\cup S_1\in \mathcal{A}_1$ and $G_2\cup S_2\in \mathcal{A}_2$, a contradiction.
\end{proof}
\begin{lemma}\label{lem: x is O(1)}
    Let $n$ be sufficiently large $($relative to $r$$)$ and $r\geq 2$. Then $\sum_{i=1}^r|\mathcal{S}_s(\mathcal{A}_i)|\geq \binom{n}{s}-3\binom{k}{s}^2$.
    \end{lemma}
    \begin{proof}
        Let $\sum_{i=1}^r|\mathcal{S}_s(\mathcal{A}_i)|= \binom{n}{s}-a$. It remains to show that $a\leq 3\binom{k}{s}^2$. If there are at least $r-1$ families in  $\{\mathcal{A}_i(S);1\leq i\leq s\}$ are empty, then we have
        \begin{align}\label{09}
            \sum_{i=1}^r|\mathcal{A}_i(S)|&\leq
            \begin{cases}
                \binom{n-s}{k-s}\text{\qquad\qquad\,\, for every }S\in \bigcup_{i=1}^r\mathcal{S}_s(\mathcal{A}_i)\\
                \frac{3}{2}\binom{n-s}{k-s}-\binom{n-k}{k-s}\text{ for every }S\in \binom{[n]}{s}\setminus\bigcup_{i=1}^r\mathcal{S}_s(\mathcal{A}_i).
            \end{cases}
        \end{align}
        Assume there are at least two families in  $\{\mathcal{A}_i(S);1\leq i\leq s\}$ are non-empty. By Lemma \ref{3.1}, $\sum_{i=1}^r|\mathcal{A}_i(S)|\leq \max\left\{\binom{n-s}{k-s}-\binom{n-k}{k-s}+r-1,r\binom{n-s-1}{k-s-1}\right\}=O(n^{k-s-1})< \min\{\binom{n-s}{k-s},\frac{3}{2}\binom{n-s}{k-s}-\binom{n-k}{k-s}\}$ for every $S\in \binom{[n]}{s}$ which implies  (\ref{09}) always holds. Now we estimate the size of $\sum_{i=1}^r|\mathcal{A}_r|$.
 \begin{align*}\label{eq:2}
               \sum_{i=1}^r|\mathcal{A}_r|&=\binom{k}{s}^{-1}\sum_{S\in \binom{[n]}{s}}\sum_{i=1}^r|\mathcal{A}_i(S)|\\
        &=\binom{k}{s}^{-1} \left( \sum_{S\in \bigcup^r_{i=1}\mathcal{S}_s(\mathcal{A}_i) }\sum_{i=1}^r|\mathcal{A}_i(S)|+\sum_{S\in \binom{[n]}{s}\setminus(\bigcup^r_{i=1}\mathcal{S}_s(\mathcal{A}_i))}\sum_{i=1}^r|\mathcal{A}_i(S)| \right)\\
        &\leq \binom{k}{s}^{-1}|\bigcup^r_{i=1}\mathcal{S}_s(\mathcal{A}_i)|\binom{n-s}{k-s}\\&+\binom{k}{s}^{-1}\left(\binom{n}{s}-|\bigcup^r_{i=1}\mathcal{S}_s(\mathcal{A}_i)|\right)\left(\frac{3}{2}\binom{n-s}{k-s}-\binom{n-k}{k-s}\right)\\
        &=\binom{k}{s}^{-1}\left(\binom{n-k}{k-s}-\frac{1}{2}\binom{n-s}{k-s}\right)\left(|\bigcup^r_{i=1}\mathcal{S}_s(\mathcal{A}_i)|\right)\\&+\binom{k}{s}^{-1}\left(\frac{3}{2}\binom{n-s}{k-s}-\binom{n-k}{k-s}\right)\binom{n}{s}.\\
        &\leq \binom{k}{s}^{-1}\left(\binom{n-k}{k-s}-\frac{1}{2}\binom{n-s}{k-s}\right)\left(\sum^r_{i=1}|\mathcal{S}_s(\mathcal{A}_i)|\right)\\&+\binom{k}{s}^{-1}\left(\frac{3}{2}\binom{n-s}{k-s}-\binom{n-k}{k-s}\right)\binom{n}{s}.
 \end{align*}
 Noticing that $\binom{n}{k}\binom{k}{s}=\binom{n-s}{k-s}\binom{n}{s}$, and combining with (\ref{seq4-1}),  we have
    \begin{align*}
        &\binom{k}{s}\binom{n}{k}-\binom{k}{s}^2\binom{n-k}{k-s}+o(n^{k-s})\\&\leq \left(\binom{n-k}{k-s}-\frac{1}{2}\binom{n-s}{k-s}\right)\left(\binom{n}{s}-a\right)+\left( \frac{3}{2}\binom{n-s}{k-s}-\binom{n-k}{k-s}\right)\binom{n}{s}\\
       &=\binom{k}{s}\binom{n}{k}-\left(\binom{n-k}{k-s}-\frac{1}{2}\binom{n-s}{k-s}\right)a.
    \end{align*}
    Then we have $$a\leq\frac{{\binom{k}{s}^2\binom{n-k}{k-s}+o(n^{k-s})}}{\binom{n-k}{k-s}-\frac{1}{2}\binom{n-s}{k-s}}\leq 3\binom{k}{s}^2$$
    when $n$ is large enough.
    \end{proof}
\begin{lemma}\label{lem: x at most k}
    Let $n$ be sufficiently large $($relative to $r$$)$ and $r\geq 2$. If  there are exactly $r-1$ families of $\mathcal{S}_s(\mathcal{A}_1),\dots,\mathcal{S}_s(\mathcal{A}_r)$ $(i.e.\,  \mathcal{S}_s(\mathcal{A}_1),\dots,\mathcal{S}_s(\mathcal{A}_{r-1}))$ are empty,  then $|\bigcup_{i=1}^{r-1}\mathcal{A}_i|= 1$.
\end{lemma}
\begin{proof} Recall $\partial_s(\mathcal{F}) = \left\{G \in \binom{[n]}{s} : G \subset F \text{ for some } F \in \mathcal{F}\right\}$.
Note that $\mathcal{S}_s(\mathcal{A}_r)\not=\emptyset$. Then $\mathcal{A}_r(S)\not=\emptyset$ if $S\in \mathcal{S}_s(\mathcal{A}_r)$.
Let $S\in \mathcal{S}_s(\mathcal{A}_r)$. Since
$$\sum^r_{i=1}|\mathcal{A}_i(S)|\ge |\mathcal{A}_r(S)|>\frac{3}{2}\binom{n-s}{k-s}-\binom{n-k}{k-s},$$
    by Lemma \ref{3.1} and $\mathcal{A}_r(S)\not=\emptyset$, we have $\mathcal{A}_i(S)=\emptyset$ for $i\in [r-1]$. Hence
        $\mathcal{S}_s(\mathcal{A}_r)\cap \partial_s(\mathcal{A}_i)= \emptyset$ for any $i\in[r-1]$ which implies $\binom{n}{s}-\sum^{r-1}_{i=1}|\mathcal{S}_s(\mathcal{A}_r)|=\binom{n}{s}-|\mathcal{S}_s(\mathcal{A}_r)|\geq |\bigcup^{r-1}_{i=1}\partial_s(\mathcal{A}_i)|$. Assume $|\partial_s(\bigcup^{r-1}_{i=1}\mathcal{A}_i)|=\binom{x}{s}$, where $x$ is a real number.  By Lemma \ref{lem: x is O(1)}, we have $\binom{n}{s}-\sum^{r-1}_{i=1}|\mathcal{S}_s(\mathcal{A}_r)|\leq 3\binom{k}{s}^2=O(1)$. Hence, $\binom{x}{s}\leq3\binom{k}{s}^2=O(1)$. By Corollary \ref{cor}, we have $|\bigcup^{r-1}_{i=1}\mathcal{A}_i|\leq \binom{x}{k}=O(1)$. For any $S\in \bigcup^{r-1}_{i=1}\partial_s(\mathcal{A}_i)$, say $S\in\partial_s(\mathcal{A}_1)$ and $S\subseteq K\in \mathcal{A}_1$, we have $\hat{\mathcal{A}}_1(S)\not=\emptyset$. Assume, without loss of generality, that $\hat{\mathcal{A}}_r(S)\not=\emptyset$. By the proof of  Claim $\ref{C_1}$, we have $\hat{\mathcal{A}}_1(S),\hat{\mathcal{A}}_r(S)$ are cross-intersecting. Then $|\mathcal{A}_r(S)|=|\hat{\mathcal{A}}_r(S)|\leq |\{F\in\binom{[n]\setminus S}{k-s}:F\cap K\neq \emptyset\}|\leq \binom{n-s}{k-s}-\binom{n-k}{k-s}$. Thus $\sum_{i=1}^r|\mathcal{A}_i(S)|\leq \sum_{i=1}^{r-1}|\mathcal{A}_i(S)|+|\mathcal{A}_r(S)|\leq (r-1)\binom{x}{k}+\binom{n-s}{k-s}-\binom{n-k}{k-s}=O(n^{k-s-1})$.
    Then
    \begin{align*}
         \sum_{i=1}^r|\mathcal{A}_i|
        &\leq  \sum_{i=1}^{r-1}|\mathcal{A}_i| +\left|\left\{F\in\mathcal{A}_r:\partial_s(F)\cap (\bigcup^{r-1}_{i=1}\partial_s(\mathcal{A}_i))=\emptyset\right\}\right|+\sum_{S\in \bigcup^{r-1}_{i=1}\partial_s(\mathcal{A}_i)}|\mathcal{A}_r(S)|\\
        &\leq (r-1)\binom{x}{k}+\left(\binom{n}{k}-(1-o(1))\binom{x}{s}\binom{n-s}{k-s}\right)
        +\binom{x}{s}O(n^{k-s-1})\\
        &=\binom{n}{k}-(1-o(1))\binom{n-s}{k-s}\binom{x}{s}.
    \end{align*}
   Combining with (\ref{seq4-1}), we have  $x\leq k$ which implies $|\bigcup_{i=1}^{r-1}\mathcal{A}_i|\leq \binom{x}{k}=1$. Since $\mathcal{A}_i\not=\emptyset$ for all $1\le i\le r-1$, we have $|\bigcup_{i=1}^{r-1}\mathcal{A}_i|=1$.
\end{proof}
\begin{lemma}\label{3.6}
     Let $n,r$ be positive integers with $n$ sufficiently large $($relative to $r)$ and $r\geq 2$. Let $L= \{0,2\}$. If non-empty families $\mathcal{A}_1,\dots,\mathcal{A}_r\subseteq \binom{[n]}{2}$ are pairwise cross $L$-intersecting, then
    $$\sum_{i=1}^r|\mathcal{A}_i|\leq \frac{(n-2)(n-3)}{2}+r,$$
    the equality holds only when $\mathcal{A}_1=\mathcal{A}_2=\dots=\mathcal{A}_{r-1}=\{[2]\}$ and $\mathcal{A}_r=\{A\in \binom{[n]}{2}:~|A\cap [2]|\in L\}$ (up to isomorphism).
\end{lemma}
\begin{proof} Assume $\sum_{i=1}^r|\mathcal{A}_i|$ is maximized. By  (\ref{seq4-1}), we  have
       $\sum_{i=1}^r|\mathcal{A}_r|\geq \frac{(n-2)(n-3)}{2}+r$. If $\mathcal{S}_1(\mathcal{A}_1),\mathcal{S}_1(\mathcal{A}_2),\ldots,\mathcal{S}_1(\mathcal{A}_r)$ are all empty, by (\ref{09}),
       \begin{align*}
           \sum_{i=1}^r|\mathcal{A}_r|=\frac{1}{2}\sum_{i\in[n]}|\mathcal{A}_r(\{i\})|
           \leq \frac{1}{2}\sum_{i\in[n]}(\frac{3}{2}(n-1)-(n-2))\leq \frac{1}{4}n^2+\frac{1}{4}n<\frac{(n-2)(n-3)}{2}+r,
       \end{align*}
    yielding a contradiction. Thus, at least one of $\mathcal{S}_1(\mathcal{A}_1),\mathcal{S}_1(\mathcal{A}_2),\ldots,\mathcal{S}_1(\mathcal{A}_r)$ is non-empty.

    We first show that there are exactly $r-1$ empty families in
     $\mathcal{S}_1(\mathcal{A}_1),\mathcal{S}_1(\mathcal{A}_2),\ldots,\mathcal{S}_1(\mathcal{A}_r)$. Otherwise, at least two of them are non-empty. Without loss of generality, assume $\mathcal{S}_1(\mathcal{A}_1)$, $\mathcal{S}_1(\mathcal{A}_2)$ are non-empty and $\{i\}\in \mathcal{S}_1(\mathcal{A}_1)$, $\{j\}\in \mathcal{S}_1(\mathcal{A}_2)$. Then, $i\neq j$ follows from Lemma \ref{3.3}. Define $\mathcal{B}_i=\left\{x\in [n]\setminus\{i,j\}:\{i,x\}\in \mathcal{A}_1\right\}$ and $\mathcal{B}_j=\left\{x\in [n]\setminus\{i,j\}:\{j,x\}\in \mathcal{A}_2\right\}$. And we have $|\mathcal{B}_i|+|\mathcal{B}_j|>2(\frac{1}{2}(n+1)-1)= n-1$. Thus, we can select $x\in \mathcal{B}_i\cap \mathcal{B}_j$ such that $|\{x,i\}\cap\{x,j\}|=1\notin L$ where $\{x,i\}\in \mathcal{A}_1$, $\{x,j\}\in \mathcal{A}_2$, a contradiction.

   By Lemma \ref{lem: x at most k}, when $n$ is large enough, we have $|\bigcup_{i=1}^{r-1}\mathcal{A}_i|=1$. It is immediately obtained that $\mathcal{A}_i$ are identical singletons for $i\in[1,r-1]$, which ends the proof.
    %Then, $x=O(1)$. Therefore,
    %\begin{align*}
        %\frac{(n-2)(n-3)}{2}+r&\leq \sum_{i=1}^r|\mathcal{A}_r|\\
        %&\leq  \sum_{i=2}^r|\mathcal{A}_r| +\left|\left\{F\in\mathcal{A}_1:\partial_s(F)\cap (\bigcup^{r-1}_{i=1}\partial_s(\mathcal{A}_i))=\emptyset\right\}\right|+\sum_{S\in \bigcup^{r-1}_{i=1}\partial_s(\mathcal{A}_i)}|\mathcal{A}_1(S)|\\
        %&\leq O(1)+\left(\binom{n}{k}-(1-o(1))x\binom{n-s}{k-s}\right)+x\left(\binom{n-s}{k-s}-\binom{n-k}{k-s}\right)\\
        %&=\binom{n}{k}-(1-o(1))\binom{n-s}{k-s}x.
    %\end{align*}
    %Therefore $x\leq 2$.
\end{proof}

\subsection{The case when $0\not\in L$}
 By Theorem \ref{t-intersecting}, we only need to consider the case $L=\{\ell_1,\dots,\ell_t\}\neq [\ell_1,k], \ell_1< \ell_2<\dots< \ell_t$ and $\ell_1\neq 0$.

First, we present a result which is helpful in the proof.
\begin{theorem}[\cite{Erdos1965APO}]\label{thm: matchings}
    For positive integers $s$ and $k$, there exist $n_0=n_0(s,k)$, such that when $n\geq n_0$, if a family $\mathcal{F}\subseteq \binom{[n]}{k}$ with $|\mathcal{F}|\geq \binom{n}{k}-\binom{n-s}{k}=D(s,k)n^{k-1}+O(n^{k-2})$ for some $D(s,k)$, then there exist at least $s$ disjoint sets in $\mathcal{F}$.
\end{theorem}

In the following discussion, we set $M=M(k,L)=\max\{ C(k,[\ell_1,k]),D(k+1,k-\ell_1)+1\}$, where $C(k,[\ell_1,k])$ is described in Theorem \ref{thm: frankl intersect}, and $D(k+1,k-\ell_1)$ is described in Theorem \ref{thm: matchings}.

\begin{lemma}\label{claim: l1 intersecting}
    Let  $\mathcal{A}_1,\dots,\mathcal{A}_r\subseteq \binom{[n]}{k}$ be pairwise cross $L$-intersecting non-empty families, where $L=\{\ell_1,\dots,\ell_t\}\neq [\ell_1,k], \ell_1< \ell_2< \dots< \ell_t$ and $\ell_1\neq 0$. Assume $|\mathcal{A}_1|\leq |\mathcal{A}_2|\leq \dots \leq |\mathcal{A}_r|$. If $n$ is large enough and $\sum_{i=1}^r|\mathcal{A}_r|\geq \sum_{i\in L}\binom{k}{i}\binom{n-k}{k-i}+r-1$, we have the following properties:

    \noindent \textbf{\rm(i)} $\mathcal{A}_i$ is $\ell_1$-intersecting for any $1\leq i\leq r-1$.

    \noindent \textbf{\rm(ii)} If $|\mathcal{A}_i|\geq M n^{k-\ell_1-1}$ for some $1\leq i\leq r-1$, then there exists a set $X$ with size $\ell_1$ such that $X\subseteq A$ for every $A\in \bigcup_{j=1}^r \mathcal{A}_j$.
\end{lemma}
\noindent\textit{Proof } (i)
       Suppose there is $ i\in [r-1]$, say $i=1$ and $A_1,A_2\in \mathcal{A}_1$ such that $|A_1\cap A_2|\leq \ell_1-1$.
    Then we have $|F\cap (A_1\cup A_2)|\geq \ell_1+1$ for every $F\in \mathcal{A}_r$. Hence
    $$|\mathcal{A}_r|\leq \sum_{i=\ell_1+1}^k\binom{|A_1\cup A_2|}{i}\binom{n-|A_1\cup A_2|}{k-i}=O(n^{k-\ell_1-1}),$$
   and then $\sum_{i=1}^{r}|\mathcal{A}_i|=O(n^{k-\ell_1-1})$, a contradiction with the assumption.

\noindent(ii)
    Assume there is $i\in[r-1]$ such that $|\mathcal{A}_i|\ge Mn^{k-\ell_1-1}$. By (i) and  Theorem \ref{thm: frankl intersect}, there exist a set $X$ with size $\ell_1$ such that $X\subseteq \cap_{A\in \mathcal{A}_i}A$.  Then $|\hat{\mathcal{A}}_i(X)|\geq (D(k+1,k-\ell_1)+1)n^{k-\ell_1-1}$. By Theorem \ref{thm: matchings}, there exist $k+1$  disjoint sets (denoted as $B_1,\dots,B_{k+1}$) in $\hat{\mathcal{A}}_i(X)$ when $n$ is large enough.

    Suppose there is $A\in \mathcal{A}_j$ with $j\in[r]\setminus\{i\}$  such that $A\cap X\neq X$. Then there is $t\in [k+1]$ such that $|A\cap (X\cup B_t)|<\ell_1$. Since $X\cup B_t\in \mathcal{A}_i$,  a contradiction with (i).\hfill$\square$
\vskip.2cm

\subsection{Proof of Theorem \ref{thm: r cross L intersecting}}
Now we complete the proof of Theorem \ref{thm: r cross L intersecting}.
\vskip.2cm

\noindent\text{\em Proof of Theorem \ref{thm: r cross L intersecting}} Let $\mathcal{A}_1,\dots,\mathcal{A}_r$ be pairwise cross $L$-intersecting families with $k\in L$ which obtain the maximum $\sum^r_{i=1}|\mathcal{A}_i|$, where $L=\{\ell_1,\dots,\ell_t\}\neq [\ell_1,k]$ with $\ell_1< \ell_2< \dots< \ell_t$.
 We complete the proof by consider two cases.

%Suppose $|\mathcal{A}_1|\leq \dots\leq |\mathcal{A}_r|$ and $\sum_{i=1}^r|\mathcal{A}_r|\geq \sum_{i\in L}\binom{k}{i}\binom{n-k}{k-i}+r-1$.

\noindent{\bf Case 1.} $\ell_1=0$.

In this case, we proceed by induction on $k$.  %Let $s$ be the minimum non-negative integer such that $s\notin L$, then $1\leq s\leq k-1$.
By Lemma \ref{3.6}, together with the trivial case when $L=[1]$ and $L=[2]$, the conclusion holds for $k=2$. Assume $k\geq 3$. Let $s$ be the minimum non-negative integer such that $s\notin L$. Then $1\le s\le k-1$.
%Let $\mathcal{A}_1, \dots, \mathcal{A}_r \subseteq \binom{[n]}{k}$ be pairwise cross $L$-intersecting families that attain the maximum value of $\sum_{i=1}^r |\mathcal{A}_i|$.
 By (\ref{seq4-1}), we have $\sum_{i=1}^r|\mathcal{A}_i|\geq \sum_{i\in L}\binom{k}{i}\binom{n-k}{k-i}+r-1$.

 We claim that there are exactly $r-1$ empty families in $\mathcal{S}_s(\mathcal{A}_1),\dots,\mathcal{S}_s(\mathcal{A}_r)$.

 If $\mathcal{S}_s(\mathcal{A}_1),\dots,\mathcal{S}_s(\mathcal{A}_r)$ are all empty, by (\ref{09})
 \begin{align*}
     \sum_{i=1}^r|\mathcal{A}_i|&=\binom{k}{s}^{-1}\sum_{S\in\binom{[n]}{s}}|\sum_{i=1}^r|\mathcal{A}_i(S)||\leq \binom{k}{s}^{-1}\binom{n}{s}\left(\frac{3}{2}\binom{n-s}{k-s}-\binom{n-k}{k-s}\right)\\
     &<\sum_{i\in L}\binom{k}{i}\binom{n-k}{k-i}+r-1,
 \end{align*}
 yielding a contradiction.

 Now, consider that there are at least two families
 in $\mathcal{S}_s(\mathcal{A}_1),\mathcal{S}_s(\mathcal{A}_2),\ldots,\mathcal{S}_s(\mathcal{A}_r)$ are non-empty. By Lemma \ref{3.3}, $\mathcal{S}_s(\mathcal{A}_1)$,$\mathcal{S}_s(\mathcal{A}_2)$, $\ldots,\mathcal{S}_s(\mathcal{A}_r)$ are pairwise cross-$[0,\max\{0,2s-k-1\}]$-intersecting, and also $[0,\max\{0,2s-k-1\}]\cup\{s\}$-intersecting. Set $L'=[0,\max\{0,2s-k-1\}]\cup\{s\}$. Note that $s-1\notin L'$ by $s\leq k-1$. Then $L'\subseteq [0,s-2]\cup\{s\}\subseteq[0,s]$.
 By induction hypothesis, we have
 \begin{equation}\label{eq:1}
     \sum_{i=1}^r|\mathcal{S}_s(\mathcal{A}_i)|\leq \binom{n}{s}-\binom{s}{s-1}\binom{n-s}{1}+r-1=\binom{n}{s}-s(n-s)+r-1.
 \end{equation}
 Set $\sum_{i=1}^r|\mathcal{S}_s(\mathcal{A}_i)|=\binom{n}{s}-y$, then we have $y\geq s(n-s)-r+1$ by (\ref{eq:1}), a contradiction with Lemma \ref{lem: x is O(1)} when $n$ is large enough. Hence there are exactly $r-1$ empty families in $\mathcal{S}_s(\mathcal{A}_1),\dots,\mathcal{S}_s(\mathcal{A}_r)$.

 Assume $\mathcal{S}_s(\mathcal{A}_i)=\emptyset$ for any $i\in[r-1]$.
    By Lemma \ref{lem: x at most k}, $|\bigcup_{i=1}^{r-1}\mathcal{A}_i|=1$, it is immediately obtained that $\mathcal{A}_i$ are identical singletons for $i\in[r-1]$ and we are done.

\noindent{\bf Case 2.}  $\ell_1\not=0$.

We first show that $|\mathcal{A}_{i}|=1$ for any $i\in[r-1]$.
  Suppose there is $i\in [r-1]$, say $i=r-1$, such that $|\mathcal{A}_{r-1}|\geq 2$. Let $A_1,A_2\in \mathcal{A}_{r-1}$.
Then we have
\begin{align*}
        |\mathcal{A}_r|&=\sum^k_{i=\ell_1}\left|\left\{F\in \binom{[n]}{k}:|F\cap A_1|,|F\cap A_2|\in L,|F\cap(A_1\cup A_2)|=i\right\}\right|\\
        &=\left|\left\{F\in \binom{[n]}{k}:|F\cap A_1|,|F\cap A_2|\in L,|F\cap(A_1\cup A_2)|=\ell_1\right\}\right|+O(n^{k-\ell_1-1})\\
        &= \binom{|A_1\cap A_2|}{\ell_1}\binom{n-|A_1\cup A_2|}{k-\ell_1}+O(n^{k-\ell_1-1})\\
        &\leq  \binom{k-1}{\ell_1}\binom{n-(k+1)}{k-\ell_1}+O(n^{k-\ell_1-1}).
    \end{align*}
Since  $\sum_{i=1}^r|\mathcal{A}_i|\geq\sum_{i\in L}\binom{k}{i}\binom{n-k}{k-i}+r-1= \binom{k}{\ell_1}\binom{n-k}{k-\ell_1}-O(n^{k-\ell_1-1})$ by (\ref{seq4-1}), we have
\begin{align*}
    \sum_{i=1}^{r-1}|\mathcal{A}_i|&\geq \left( \binom{k}{\ell_1}-\binom{k-1}{\ell_1} \right)\binom{n-(k+1)}{k-\ell_1}-O(n^{k-\ell_1-1})\\
    &=\binom{k-1}{\ell_1-1}\binom{n-(k+1)}{k-\ell_1}-O(n^{k-\ell_1-1}).
\end{align*}

 For a fixed constant $M\geq \max\{C(k,[\ell_1,k]),D(k+1,k-\ell_1)+1\}$, where $C(k,[\ell_1,k])$ is described in Theorem \ref{thm: frankl intersect} and $D(k+1,k-\ell_1)$ is described in Theorem \ref{thm: matchings}, $|\mathcal{A}_i|\geq Mn^{k-\ell_1-1}$ holds for some $i\in[1,r-1]$. By Lemma \ref{claim: l1 intersecting} (ii), there exists a set $X$ with size $\ell_1$ such that $X\subseteq A$ for every $A\in \bigcup_{j=1}^r \mathcal{A}_j$. Define $\mathcal{A}_i^0=\{F\setminus X:F\in \mathcal{A}_i\}$ and $L'=\{x-\ell_1:x\in L\}$.
Then  $\mathcal{A}_1^0,\dots,\mathcal{A}_r^0$ are $(k-\ell_1)$-uniform pairwise cross $L'$-intersecting families and $0\in L'$. By Case 1, we have
\begin{align*}
    \sum_{i=1}^r|\mathcal{A}_i|=\sum_{i=1}^r|\mathcal{A}_i^0|&\leq\sum_{i\in L'}\binom{k-\ell_1}{i}\binom{n-k}{k-\ell_1-i}+(r-1)\\
    &<\sum_{i\in L}\binom{k}{i}\binom{n-k}{k-i}+r-1,
\end{align*}
 a contradiction with (\ref{seq4-1}).

Thus we have $|\mathcal{A}_i|=1$ for $i\in [r-1]$. When $r\geq 3$, we may assume $\mathcal{A}_1=\dots=\mathcal{A}_{r-1}$, otherwise we can similarly have $|\mathcal{A}_r|\leq \binom{k-1}{\ell_1}\binom{n-(k+1)}{k-\ell_1}+O(n^{k-\ell_1-1})$ and then
$$\sum_{i=1}^r|\mathcal{A}_i|\leq \binom{k-1}{\ell_1}\binom{n-(k+1)}{k-\ell_1}+O(n^{k-\ell_1-1})<\binom{k}{\ell_1}\binom{n-k}{k-\ell_1}-O(n^{k-\ell_1-1}),$$
a contradiction with (\ref{seq4-1}).

Hence, we obtain that $\mathcal{A}_1=\dots=\mathcal{A}_{r-1}=\{A\}$, which completes the proof. \hfill$\square$

\section{$r$-cross $L$-intersecting families}
\noindent\text{\em Proof of Theorem \ref{thm: cross interval intersecting}} Let $\mathcal{A}_1,\dots,\mathcal{A}_r\subseteq\binom{[n]}{k}$ be a non-empty $r$-cross $[\ell,s-1]$-intersecting families that attain the maximum value of $\sum_{i=1}^r|\mathcal{A}_i|$, where $s\in[k]$.
Let $\mathcal{A}_1^0=\{[k]\}$, $\mathcal{A}_2^0=\{A\in\binom{[n]}{k}:|A\cap [k]|\leq s-1, [\ell]\subseteq A\}$ and $\mathcal{A}_3^0=\dots=\mathcal{A}_r^0=\{A\in \binom{[n]}{k}:[\ell]\subseteq A\}$ (if $\ell=0$, let $[\ell]=\emptyset$). It is easy to check $\mathcal{A}_1^0,\dots,\mathcal{A}_r^0$ are $r$-cross $[\ell,s-1]$-intersecting. Thus we have
\begin{align}\label{e4-1}
\sum_{i=1}^r|\mathcal{A}_i|\geq \sum_{i=1}^r|\mathcal{A}_i^0|=(r-1)\binom{n-\ell}{k-\ell}-\sum_{i=s-\ell}^{k-\ell}\binom{k-\ell}{i}\binom{n-\ell-k}{k-\ell-i}+1.\end{align}
We will divide the proof into two cases, when $\ell=0$ and $\ell>0$. Let $L=[\ell,s-1]$.

\noindent{\bf Case 1.} $\ell=0$.

In this case, we will prove $$\sum_{i=1}^r|\mathcal{A}_i|\leq (r-1)\binom{n}{k}-\sum_{i=s}^{k}\binom{k}{i}\binom{n-k}{k-i}+1.$$
We prove it by induction on $k$. When $k=1$, we have $L=\{0\}$. Since every $i\in[n]$ is contained in at most $r-1$ families of $\mathcal{A}_1,\dots,\mathcal{A}_r$,  we have $$\sum_{i=1}^r|\mathcal{A}_i|\leq (r-1)n=(r-1)\binom{n}{1}-\binom{1}{1}\binom{n-1}{1-1}+1,$$ and we are done. Then we assume $k\geq 2$, and the result holds for every $\ell\leq k-1$.

Let $t\in [k]$,  $S\in \binom{[n]}{t}$ and  $\mathcal{F}\subseteq\binom{[n]}{k}$. Recall $\mathcal{F}(S)=\{F\in \mathcal{F}:S\subseteq F\}$ and $\hat{\mathcal{F}}(S)=\{F\setminus S:F\in \mathcal{F}(S)\}$. By the definition of $r$-cross $[0,s-1]$-intersecting, for every $S\in \binom{[n]}{s}$, at least one of $\mathcal{A}_1(S),\dots,\mathcal{A}_r(S)$ is empty which implies $\sum_{i=1}^r|\mathcal{A}_i(S)|\leq (r-1)\binom{n-s}{k-s}$.
We set
$$\mathcal{T}_s(\mathcal{F})=\left\{T\in \binom{[n]}{s}:|\mathcal{F}(T)|\geq \binom{n-s}{k-s}-\binom{n-k}{k-s}\right\},$$
and  $\mathcal{S}=\left\{ S\in \binom{[n]}{s}: \sum_{i=1}^r|\mathcal{A}_i(S)|\geq (r-1)\binom{n-s}{k-s}-\binom{n-k}{k-s}\right\}$.
Then we have the following result.
\begin{claim}\label{lm-4-1}
    At least one of $\mathcal{T}_s(\mathcal{A}_1),\dots,\mathcal{T}_s(\mathcal{A}_r)$ is empty, and $|\mathcal{S}|\geq \binom{n}{s}-\binom{k}{s}^2.$
\end{claim}
\noindent{\bf Proof of Claim \ref{lm-4-1}}
    First,  we have the following inequality.
    \begin{align*}
    %(r-1)\binom{n}{k}-\binom{k}{s}\binom{n-k}{k-s}+O(n^{k-s-1})&\leq \\
    \sum_{i=1}^r|\mathcal{A}_i|&\leq \binom{k}{s}^{-1}\sum_{S\in \mathcal{S}}\sum_{i=1}^r|\mathcal{A}_i(S)|+\binom{k}{s}^{-1}\sum_{S\in\binom{[n]}{s}\setminus \mathcal{S}}\sum_{i=1}^r|\mathcal{A}_i(S)|\\
    &\leq (r-1)\binom{k}{s}^{-1}\binom{n-s}{k-s}|\mathcal{S}|\\&+\binom{k}{s}^{-1}\left((r-1)\binom{n-s}{k-s}-\binom{n-k}{k-s} \right)\left( \binom{n}{s}-|\mathcal{S}|\right ).
    \end{align*}
    Combining with (\ref{e4-1}), we have
        $$\binom{n-k}{k-s}|\mathcal{S}|\geq \binom{n}{s}\binom{n-k}{k-s}-\binom{k}{s}^2\binom{n-k}{k-s},$$
    which implies $|\mathcal{S}|\geq \binom{n}{s}-\binom{k}{s}^2$.

    Let $S\in\mathcal{S}$. Then at least one of $\mathcal{T}_s(\mathcal{A}_1),\dots, \mathcal{T}_s(A_r)$ does not contain $S$ by $\mathcal{A}_1,\dots,\mathcal{A}_r$ being $r$-cross $[0,s-1]$-intersecting. Suppose $S$ is contained in at most $r-2$ families of $\mathcal{T}_s(\mathcal{A}_1),\dots, \mathcal{T}_s(A_r)$. Then we have $$\sum_{i=1}^r|\mathcal{A}_i(S)|\leq (r-2)\binom{n-s}{k-s}+\binom{n-s}{k-s}-\binom{n-k}{k-s}-1<(r-1)\binom{n-s}{k-s}-\binom{n-k}{k-s},$$
    a contradiction with $S\in\mathcal{S}$. Hence  for every $S\in\mathcal{S}$, $S$ is contained in exactly $r-1$ families of $\mathcal{T}_s(\mathcal{A}_1),\dots, \mathcal{T}_s(A_r)$.
    Thus we have  \begin{align}\label{e4-2}
    \sum_{i=1}^r|\mathcal{T}_s(A_i)|\geq (r-1)|\mathcal{S}|\geq (r-1)\binom{n}{s}-(r-1)\binom{k}{s}^2.\end{align}
   % Moreover, $|\bigcup_{i=1}^r\mathcal{T}_s(\mathcal{A}_i)|\geq |\mathcal{S}|\geq \binom{n}{s}-\binom{k}{s}^2$.
    Note that $\mathcal{T}_s(\mathcal{A}_1),\dots,\mathcal{T}_s(\mathcal{A}_r)\subseteq\binom{[n]}{s}$ are $r$-cross $[0,s-1]$-intersecting. If $\mathcal{T}_s(\mathcal{A}_i)\not=\emptyset$ for any $i\in [r]$,  then by induction hypothesis, we have
    $$\sum_{i=1}^r|\mathcal{T}_s(\mathcal{A}_i)|\leq (r-1)\binom{n}{s}-\binom{k}{s}\binom{n-k}{k-s}+1<(r-1)\binom{n}{s}-(r-1)\binom{k}{s}^2,$$
    a contradiction with (\ref{e4-2}). Thus at least one of $\mathcal{T}_s(\mathcal{A}_1),\dots,\mathcal{T}_s(\mathcal{A}_r)$ is empty.\q

 By Claim \ref{lm-4-1},
we may assume $\mathcal{T}_s(\mathcal{A}_1)$ is empty. Then for any $S\in \binom{[n]}{s}$, $|\mathcal{A}_1(S)|\leq \binom{n-s}{k-s}-\binom{n-k}{k-s}-1$.%We first have the following claim.

\begin{claim}\label{D_1}
$\mathcal{S}\cap \partial_s(\mathcal{A}_1)=\emptyset$.
\end{claim}
\noindent{\bf Proof of Claim \ref{D_1}} Suppose there is $S\in \mathcal{S}\cap \partial_s(\mathcal{A}_1)$. Then $\mathcal{A}_1(S)\not=\emptyset$. Since at least one of $\mathcal{A}_1(S),\dots,\mathcal{A}_r(S)$ is empty, we have
$$\sum_{i=1}^r|\mathcal{A}_i(S)|\leq (r-2)\binom{n-s}{k-s}+\binom{n-s}{k-s}-\binom{n-k}{k-s}-1<(r-1)\binom{n-s}{k-s}-\binom{n-k}{k-s},$$
    a contradiction with $S\in\mathcal{S}$.\q

By Claims \ref{D_1} and  \ref{lm-4-1},  $|\partial_s(\mathcal{A}_1)|\leq \binom{k}{s}^2$. Assume $|\partial_s(\mathcal{A}_1)|=\binom{x}{s}$, where $x$ is a real number. Then $x=O(1)$ and by Corollary \ref{cor}, we have $|\mathcal{A}_1|\leq \binom{x}{k}$. For $i=2,\dots,r$,
 set $\mathcal{B}_i=\left\{ F\in\mathcal{A}_i:\partial_s(F)\cap\partial_s(\mathcal{A}_1)=\emptyset \right\}$.
  Since $|\partial_s(\mathcal{A}_1)|=O(1)$, we have $|\partial_1(\mathcal{A}_1)|=O(1)$. Let $A_1'=[n]\setminus \partial_1(\mathcal{A}_1)$. Then $|A_1'|=n-O(1)$. For every $S\in \partial_s(\mathcal{A}_1)$, choose $B\in \binom{A_1'}{k-s}$, we have $S\cup B\in \overline{\mathcal{B}_i}$ for all $i=2,\dots,r$, where $\overline{\mathcal{B}_i}=\binom{[n]}{k}\setminus \mathcal{B}_i$. Thus for $i=2,\dots,r$, $$|\overline{\mathcal{B}_i}|\geq |\partial_s(\mathcal{A}_1)|\binom{n-|\partial_1(\mathcal{A}_1)|}{k-s}=(1-o(1))\binom{x}{s}\binom{n-s}{k-s}.$$
Then for $i=2,\dots,r$, we have $$|\mathcal{B}_i|\leq \binom{n}{k}-(1-o(1))\binom{x}{s}\binom{n-s}{k-s}.$$
  And for every $S\in \partial_s(\mathcal{A}_1)$, at least one of $\mathcal{A}_2(S),\dots,\mathcal{A}_r(S)$ is empty. Hence $\sum_{i=2}^r\mathcal{A}_i(S)\leq (r-2)\binom{n-s}{k-s}$.
Thus%We have the following inequalities.
\begin{align*}
    \sum_{i=1}^r|\mathcal{A}_i|
    &\leq |\mathcal{A}_1|+\sum_{i=2}^r|\mathcal{B}_i|+\sum_{S\in \partial_s(\mathcal{A}_1)}\sum_{i=2}^r|\mathcal{A}_r(S)|\\
    &\leq \binom{x}{k}+(r-1)\left(\binom{n}{k}-(1-o(1))\binom{x}{s}\binom{n-s}{k-s} \right)+(r-2)\binom{x}{s}\binom{n-s}{k-s}\\
    &=(r-1)\binom{n}{k}-(1-o(1))\binom{x}{s}\binom{n-s}{k-s}+O(1).
\end{align*}
    By (\ref{e4-1}) and $\ell=0$, $\sum_{i=1}^r|\mathcal{A}_i|\geq (r-1)\binom{n}{k}-\binom{k}{s}\binom{n-k}{k-s}-o(n^{k-s})$.
    Then we have $x\leq k$ when $n$ is large enough. Since $\mathcal{A}_1$ is not empty, it's immediately obtained that $\mathcal{A}_1$ is identical singletons. We may assume $\mathcal{A}_1=\{[k]\}$.
    For $i=2,\dots,r$, we set $$\mathcal{C}_i^{s-1}=\left\{ A\in \mathcal{A}_i: |A\cap[k]|\leq s-1  \right\},$$ and for every integer $t\in[s,k]$, set
    $$\mathcal{C}_i^t=\left\{A\in \mathcal{A}_i:|A\cap [k]|=t \right\}.$$
    Then $\mathcal{A}_i=\bigcup_{t=s-1}^k\mathcal{C}_i^t$ for $i=2,\dots r$. And we have $$\sum_{i=1}^r|\mathcal{C}_i^{s-1}|\leq (r-1)\left(\binom{n}{k}-\sum_{i=s}^k\binom{k}{i}\binom{n-k}{k-i}\right).$$
    Since for every set $T\subseteq[k]$ with $|T|=t\geq s$, at least one of $\mathcal{A}_2(T),\dots,\mathcal{A}_r(T)$ is empty, we have
    $$\sum_{t=s}^k\sum_{i=2}^r|\mathcal{C}_i^t|\leq (r-2)\sum_{i=s}^k\binom{k}{i}\binom{n-k}{k-i}.$$
    Finally, we have
    \begin{align*}
        \sum_{i=1}^r|\mathcal{A}_i|&=1+\sum_{i=2}^r|\mathcal{C}_i^{s-1}|+\sum_{i=2}^r\sum_{t=s}^k|\mathcal{C}_i^t|\\
        &\leq 1+(r-1)\left(\binom{n}{k}-\sum_{i=s}^k\binom{k}{i}\binom{n-k}{k-i} \right)+(r-2)\sum_{i=s}^k\binom{k}{i}\binom{n-k}{k-i}\\
        &=(r-1)\binom{n}{k}-\sum_{i=s}^k\binom{k}{i}\binom{n-k}{k-i}+1,
    \end{align*}
  and we are done.

  \noindent{\bf Case 2.} $\ell>0$.

    In this case,  we will prove
    $$\sum_{i=1}^r|\mathcal{A}_i|\leq(r-1)\binom{n-\ell}{k-\ell}-\sum_{i=s-\ell}^{k-\ell}\binom{k-\ell}{i}\binom{n-k}{k-\ell-i}+1.$$
    We may assume $|\mathcal{A}_1|\leq \dots\leq |\mathcal{A}_r|$.
We need the following results.
\begin{claim}\label{lem: ell intersecting}
    $\mathcal{A}_i$ is $\ell$-intersecting for all $i\in[r-1]$, and  the families $\mathcal{A}_1,\dots,\mathcal{A}_{r-1}$ are pairwise cross $\ell$-intersecting.
\end{claim}
\noindent{\bf Proof of Claim \ref{lem: ell intersecting}}
 Suppose there is $i\in [r-1]$ with $A_1,A_2\in \mathcal{A}_i$ or there are $i,j\in [r-1]$ with $i\not=j$ and $A_1\in \mathcal{A}_i$ and $A_2\in \mathcal{A}_j$
    such that $|A_1\cap A_2|\leq \ell-1$. Since for every set $A\in \mathcal{A}_r$, $|A\cap A_1\cap A_2|\geq \ell$,  we have $|A\cap (A_1\cup A_2)|\geq \ell+1$. Then
    $$|\mathcal{A}_r|= \sum_{i=\ell+1}^k\binom{|A_1\cup A_2|}{i}\binom{n-|A_1\cup A_2|}{k-i}=O(n^{k-\ell-1}).$$
    Then $$\sum_{i=1}^r|\mathcal{A}_i|\leq r|\mathcal{A}_r|=O(n^{k-\ell-1}),$$
    a contradiction with (\ref{e4-1}).\q

\begin{claim}\label{lem: X in all set}
    There exists a set $X\in \binom{[n]}{\ell}$ such that $X\subseteq\bigcap_{A\in \mathcal{A}_i} A$ for every $i\in [r]$.
\end{claim}

\noindent{\bf Proof of Claim \ref{lem: X in all set}}
    By Claim \ref{lem: ell intersecting}, for every two sets $A_1,A_2\in \bigcup_{i=1}^{r-1}\mathcal{A}_i$, we have $|A_1\cap A_2|\geq \ell$. And for every $A\in \mathcal{A}_r$, $|A\cap (A_1\cup A_2)|\geq \ell$. So we have $|\mathcal{A}_r|\leq \sum_{i=\ell}^k\binom{|A_1\cup A_2|}{i}\binom{n}{k-i}=\binom{n}{k-\ell}+O(n^{k-\ell-1})$. Thus, by (\ref{e4-1}), $$\sum_{i=1}^{r-1}|\mathcal{A}_i|\geq (r-2)\binom{n-\ell}{k-\ell}-O(n^{k-\ell-1}).$$
Then, for every fixed constant $M\geq \max\{ C(k,L),D(k+1,k-\ell)+1\}$, where  $C(k,L)$ is the constant in Theorem \ref{thm: frankl intersect} and $D(k+1,k-\ell)$ is described in Theorem \ref{thm: matchings}, when $n$ is large enough, there exists $i\in [r-1]$ with $|\mathcal{A}_i|\geq Mn^{k-\ell-1}$. According to Claim \ref{lem: ell intersecting}, $\mathcal{A}_i$ is $\ell$-intersecting. By Theorem \ref{thm: frankl intersect}, there exists $X$ with $|X|=\ell$ such that $X\subseteq \bigcap_{A\in \mathcal{A}_i} A$. By Theorem \ref{thm: matchings}, we can find $k+1$ disjoint sets, denote as $B_1,B_2,\dots,B_{k+1}$, in $\hat{A}_i(X)$.

Suppose there is $A\in \mathcal{A}_j$ with $j\neq i$ such that $A\cap X\neq X$. Then there is  $B_t\in \{B_1,B_2,\dots,B_{k+1}\}$ that is disjoint with $A$ such that $B_t\cup X\in \mathcal{A}_i$ and $|(B_t\cup X)\cap A|=|X\cap A|\leq \ell-1$, a contradiction. \q

By Claim \ref{lem: X in all set}, there exist $X\in\binom{[n]}{\ell}$ and $X\subseteq A$ for all $A\in \mathcal{A}_i$ and all $i\in [r]$. For $i\in [r]$, we set $\mathcal{D}_i=\{A\setminus  X: A \in\mathcal{A}_i\}$. Then $\mathcal{D}_1,\dots,\mathcal{D}_r\subseteq\binom{[n]\setminus X}{k-\ell}$ are $r$-cross $[0,s-\ell-1]$-intersecting. By  Case 1, we have
$$\sum_{i=1}^r|\mathcal{A}_i|=\sum_{i=1}^r|\mathcal{D}_i|\leq (r-1)\binom{n-\ell}{k-\ell}-\sum_{i=s-\ell}^{k-\ell}\binom{k-\ell}{i}\binom{n-k}{k-\ell-i},$$
and we are done.\hfill$\square$

\section{Open problems}
In this paper, we  completely resolved the maximum sum problem for cross $L$-intersecting families. However, for pairwise cross $L$-intersecting families, we have only addressed the case where $n$ is sufficiently large and either $k \in L$ or $L = [0, k-1]$. Naturally, this leads to the following two open questions.

\noindent\textbf{Problem 1: } Let $r\ge 3$ and $\mathcal{A}_1, \dots, \mathcal{A}_r \subseteq \binom{[n]}{k}$ be non-empty pairwise cross $L$-intersecting families. For $k\in L\neq[t,k]$ and $n\geq k$, determine $\max\sum^r_{i=1}|\mathcal{A}_i|$ and characterize the extremal structure.

\noindent\noindent\textbf{Problem 2: } Let $r\ge 3$ and $\mathcal{A}_1, \dots, \mathcal{A}_r \subseteq \binom{[n]}{k}$ be non-empty pairwise cross $L$-intersecting families. For $k\notin L\neq [0,k-1]$ and sufficiently large $n$, determine $\max\sum^r_{i=1}|\mathcal{A}_i|$ and characterize the extremal structure.

Regarding $r$-cross $L$-intersecting families, much less is known. We thus pose the following additional question.

\noindent\noindent\textbf{Problem 3: } Let $r\ge 3$ and $\mathcal{A}_1, \dots, \mathcal{A}_r \subseteq \binom{[n]}{k}$ be non-empty $r$-cross $L$-intersecting families. For $L \subseteq [0, k]$ where $L$ is not an interval, determine $\max\sum^r_{i=1}|\mathcal{A}_i|$ and characterize the extremal structure.

One may also consider the maximum product problem for cross $L$-intersecting families. When the $L$-intersecting condition corresponds to $L = [t, k]$, this is a classical problem known as a conjecture of Tokushige. The latest results are shown below.
\begin{theorem}[\cite{ZHANG2025}]
    Let \( n, k, t \) be positive integers such that \( 3 \leq t \leq k \leq n \) and \( n \geq (t+1)(k-t+1) \). If $\mathcal{A} \subseteq \binom{[n]}{k}$ and $\mathcal{B}\subseteq \binom{[n]}{k}$ are cross-\( t \)-intersecting, then $$|\mathcal{A}| |\mathcal{B}| \leq \binom{n-t}{k-t}^2.$$
    Moreover, equality holds if and only if \(\mathcal{A} = \mathcal{B}\) is a maximum \( t \)-intersecting subfamily of $\binom{[n]}{k}$.
\end{theorem}
A natural question is to generalize cross $t$-intersecting  to cross $L$-intersecting.

\noindent\noindent\textbf{Problem 4: } Let $\mathcal{A},\mathcal{B} \subseteq \binom{[n]}{k}$ be cross $L$-intersecting families. Determine $|\mathcal{A}| |\mathcal{B}|$ and characterize the extremal structure.

It is worth noting that if $L=[t,k]$ in Problem 3, the extremal configuration consists of two identical maximum $L$-intersecting families. However, for general $L$, $\mathcal{A}$ and
$\mathcal{B}$ are not necessarily equal to obtain the maximum. For example, assume $L=\{\ell_1,\ell_2,\ldots,\ell_m\}\subseteq[0,k]$. If $\mathcal{A}=\mathcal{B}$ are cross $L$-intersecting, $\mathcal{A},\mathcal{B}$ must be $L$-intersecting. Thus by Theorem \ref{thm: frankl intersect}, $|\mathcal{A}||\mathcal{B}|\leq (\prod_{1\leq i
\leq m} \frac{n - \ell_i}{k - \ell_i})^2\leq O(n^{2m})$. However, consider the cross $L$-intersecitng families $\mathcal{A}'=\{[1,k]\}$ and $\mathcal{B}'=\{F\in\binom{[n]}{k}:|F\cap[1,k]|=\ell_1\}$, we have $|\mathcal{A}'||\mathcal{B}'|=O(n^{k-\ell_1})$. If $k-\ell_1> 2m$, $|\mathcal{A}'||\mathcal{B}'|>O(n^{2m})$. Thus, $\mathcal{A}=\mathcal{B}$ can not be extremal structure. Therefore, this problem presents significant challenges in conjecturing the extremal structures, but we believe it is a highly interesting and worthwhile question.

\section*{Acknowledgement}
This work is supported by the National Natural Science Foundation of China (Grant 12571372).

\section*{Declaration of competing interest}
The authors declare that they have no known competing financial interests or personal relationships that could have appeared to influence the work reported in this paper.

\section*{Data availability}
No data was used for the research described in the article.


\begin{thebibliography}{10}
	\expandafter\ifx\csname urlstyle\endcsname\relax
	\providecommand{\doi}[1]{doi:\discretionary{}{}{}#1}\else
	\providecommand{\doi}{doi:\discretionary{}{}{}\begingroup
		\urlstyle{rm}\Url}\fi
	
	\bibitem{articlefrankl1978}
	M.~Deza, P.~Erd\H{o}s, and P.~Frankl.
	\newblock Intersection properties of systems of finite sets.
	\newblock \emph{Proceedings of The London Mathematical Society}, 36:369--384,
	03 1978.
	
	\bibitem{erdos1961intersection}
	P.~Erd\H{o}s.
	\newblock Intersection theorems for systems of finite sets.
	\newblock \emph{The Quarterly Journal of Mathematics Oxford Second Series},
	12:313--320, 1961.
	
	\bibitem{Erdos1965APO}
	P.~Erd\H{o}s.
	\newblock A problem on independent $r$-tuples.
	\newblock \emph{Ann. Univ. Sci. Budapest. E{\"o}tv{\"o}s Sect. Math},
	8(2):93-95, 1965.
	
	\bibitem{frankl2016invitation}
	P.~Frankl and N.~Tokushige.
	\newblock Invitation to intersection problems for finite sets.
	\newblock \emph{Journal of Combinatorial Theory, Series A}, 144:157--211, 2016.
	
	\bibitem{gupta2023r}
	P.~Gupta, Y.~Mogge, S.~Piga, and B.~Sch{\"u}lke.
	\newblock  $r$-cross  $t$-intersecting families via necessary intersection
	points.
	\newblock \emph{Bulletin of the London Mathematical Society}, 55(3):1447--1458,
	2023.
	
	\bibitem{hilton1977intersection}
	A.~Hilton.
	\newblock An intersection theorem for a collection of families of subsets of a
	finite set.
	\newblock \emph{Journal of the London Mathematical Society}, 2(3):369--376,
	1977.
	
	\bibitem{hilton1967some}
	A.~J. Hilton and E.~C. Milner.
	\newblock Some intersection theorems for systems of finite sets.
	\newblock \emph{The Quarterly Journal of Mathematics}, 18(1):369--384, 1967.
	
	\bibitem{HUANG2025105981}
	Y.~Huang and Y.~Peng.
	\newblock Non-empty pairwise cross-intersecting families.
	\newblock \emph{Journal of Combinatorial Theory, Series A}, 211:105981, 2025.
	
	\bibitem{Katona}
	G.~Katona.
	\newblock A theorem of finite sets.
	\newblock In \emph{Classic Papers in Combinatorics}, pages 381--401. Springer,
	1968.
	
	\bibitem{kruskal}
	J.~B. Kruskal.
	\newblock \emph{12. The Number of Simplices in a Complex}, pages 251--278.
	\newblock University of California Press, Berkeley, 2024.
	
	\bibitem{LI2025105960}
	A.~Li and H.~Zhang.
	\newblock On non-empty cross-$t$-intersecting families.
	\newblock \emph{Journal of Combinatorial Theory, Series A}, 210:105960, 2025.
	
	\bibitem{lovasz1979combinatorial}
	L.~Lov\'{a}sz.
	\newblock Combinatorial problems and exercises.
	\newblock \emph{New York}, 1979.
	
	\bibitem{shi2022}
	C.~Shi, P.~Frankl, and J.~Qian.
	\newblock On non-empty cross-intersecting families.
	\newblock \emph{Combinatorica}, 42, 09 2022.
	
	\bibitem{WANG2011455}
	J.~Wang and H.~Zhang.
	\newblock Cross-intersecting families and primitivity of symmetric systems.
	\newblock \emph{Journal of Combinatorial Theory, Series A}, 118(2):455--462,
	2011.
	\bibitem{ZHANG2025}
        J.~Wang and H.~Zhang.
        \newblock On a conjecture of Tokushige for cross-$t$-intersecting families,
        \newblock \emph{Journal of Combinatorial Theory, Series B},
        17:,49--70,
        2025.

	\bibitem{WANG2013129}
	J.~Wang and H.~Zhang.
	\newblock Nontrivial independent sets of bipartite graphs and
	cross-intersecting families.
	\newblock \emph{Journal of Combinatorial Theory, Series A}, 120(1):129--141,
	2013.
	\newblock ISSN 0097-3165.
	
	\bibitem{ZHANG2024103968}
	M.~Zhang and T.~Feng.
	\newblock A note on non-empty cross-intersecting families.
	\newblock \emph{European Journal of Combinatorics}, 120:103968, 2024.
	
\end{thebibliography}
\end{document}